\documentclass[final,5p,twocolumn]{elsarticle}%
\usepackage{amssymb, amsmath,accents}
\journal{Automatica}
\usepackage{overpic}
\usepackage{float}

\usepackage{hyperref}  

\usepackage{xcolor}
\usepackage{pdfsync}

\usepackage{graphicx}
\graphicspath{{figures/}}
\usepackage{epsfig}
\usepackage{amsmath}
\usepackage{amssymb}
\usepackage{amsthm}
\usepackage{mathtools}
\usepackage{amsfonts}
\usepackage{lineno}
\usepackage{tikz}
\usetikzlibrary{positioning,shapes}
\usetikzlibrary{arrows}

\allowdisplaybreaks

\newtheorem{theorem}{Theorem}
\newtheorem{definition}[theorem]{Definition}

\newtheorem{lemma}[theorem]{Lemma}
\newtheorem{assumption}[theorem]{Assumption}
\newtheorem{proposition}[theorem]{Proposition}
\newtheorem{remark}[theorem]{Remark}

\newtheorem{example}[theorem]{Example}

\newcommand{\norm}[1]{\left\Vert #1\right\Vert}
\newcommand{\abs}[1]{\left|#1\right|}

\newcommand{\tm}{\times}%
\newcommand{\trn}{^{\scriptscriptstyle \top}}%

	\newcommand \Iff   {\Leftrightarrow}
	\newcommand \qiq   {\quad\Iff\quad}

\usepackage{csquotes}  
\newcommand\q{\enquote}

\def\field#1{\mathbb #1}%
\def\R{\field{R}}%

\newcommand{\ep}{\varepsilon}%
\newcommand{\N}{\mathbb{N}}%
\newcommand{\KC}{\mathcal{K}}%
\newcommand{\UC}{\mathcal{U}}%
\newcommand{\A}{\mathcal{A}}%
\DeclareMathOperator{\id}{id}

\def\K{\mathcal{K}}%
\def\LL{\mathcal{L}}%
\def\KL{\mathcal{KL}}%
\def\Kinf{\mathcal{K}_\infty}%
\let\ol=\overline%
\let\ul=\underline%

\newcommand{\red}[1]{{\color{red} #1}}

\newcommand \sgn   {\mathrm{sgn}}

\usepackage{xspace}	

\title{\huge A relaxed small-gain theorem for infinite networks}

\author[First]{Navid~Noroozi}
\address[First]{Institute of Informatics, LMU Munich, 80538 Munich, Germany}
\ead{navid.noroozi@lmu.de}

\author[Second]{Andrii~Mironchenko}
\address[Second]{Faculty of Computer Science and Mathematics, University of Passau, 94032 Passau, Germany}
\ead{andrii.mironchenko@uni-passau.de}

\author[Third]{Fabian~R.~Wirth}
\address[Third]{Faculty of Computer Science and Mathematics, University of Passau, 94032 Passau, Germany}
\ead{fabian.lastname@uni-passau.de}

\cortext[corauth]{The work of N. Noroozi is supported by the DFG through the grant WI 1458/16-1. The work of A.~Mironchenko is supported by the DFG through the grant MI 1886/2-1.}

\begin{document}

\begin{abstract}

\vspace{-0.3cm}
Motivated by the scalability problem in large networks, we study stability of a network of infinitely many finite-dimensional subsystems.
We develop a so-called \emph{relaxed} small-gain theorem for input-to-state stability (ISS) with respect to a closed set and show that every exponentially input-to-state stable system \emph{necessarily} satisfies the proposed small-gain condition.
Following our bottom-up approach, we study the well-posedness of the interconnection based on the behavior of the individual subsystems.
Finally, we over-approximate large-but-finite networks by infinite networks  and show that all the stability properties and the performance indices obtained for the infinite system can be transferred to the original finite one if each subsystem of the infinite network is individually ISS.
Interestingly, the size of the truncated network does \emph{not} need to be known.
The effectiveness of our small-gain theorem is verified by application to an urban traffic network. 

\end{abstract}

\begin{keyword}
Networked systems, input-to-state stability, small-gain theorem, Lyapunov methods
\end{keyword}

\maketitle

\section{Introduction}

Recent advances in large-scale computing, cheap distributed sensing and large-scale data management have created the potential for smart applications, in which large numbers of dispersed agents need to be regulated for a common objective.
In the domain of Smart Cities, for instance, city-wide traffic control based on cheap personal communication, car-to-car communication and the deployment of numerous sensors can provide a major step toward safe, energy-efficient and environmentally friendly traffic concepts.
The vision of safe and efficient control of such large, dispersed systems requires tools that are capable of handling \emph{uncertain} and \emph{time-varying} number of participating
agents, limited communication, the need for \emph{scale-free} methods, as well as stringent \emph{safety} specifications.

Standard tools in the literature of stability analysis and control do not scale well to such increasingly common smart networked systems.
In fact, networks designed using classic tools may lead to \emph{fragile} systems, where stability and performance indices of the system do depend on the size of the systems in a way that the network tends to instability as the size of the network grows; cf. e.g.~\cite{DMS19a,Sarkar.2018,BeJ17,Jovanovic.2005b}.
An efficient approach to address this fragility is to over-approximate a finite-but-large network with an \emph{infinite network} consisting of countably  many subsystems~\cite{Bamieh.2012,Barooah.2009,BPD02}.
By treating this over-approximating network, we envisage to develop \emph{scale-free} tools for analysis and control of the original (i.e. finite-but-large) networks.
In particular, it is expected that an infinite network {captures} the essence of its corresponding finite network, i.e. the performance/stability indices achieved for the infinite network will be transferable to any finite truncation of the network; cf., e.g., a vehicle platooning application in~\cite{Jovanovic.2005b}.
However, this expectation has to be rigorously checked as \emph{counter-intuitive} results for optimal control of vehicle platoons have been already observed in~\cite{Curtain.2010}.

Motivated by all the above discussions, this paper investigates input-to-state stability (ISS) with respect to closed sets of discrete-time infinite networks within a small-gain framework.
As an infinite network is an infinite dimensional system, the \emph{set} stability problem raises well-posedness issues of the system, which are addressed first.
Then we develop so-called \emph{relaxed} small-gain conditions for which, in contrast with classic small-gain conditions~\cite{Dashkovskiy.2010,DRW07,JMW96}, every subsystem does \emph{not} have to be individually ISS.
In fact, relaxed small-gain conditions allow to treat the case that
subsystems have a stabilizing effect on each other and subsystems can be even individually \emph{unstable}; see~\cite{Geiselhart.2015,Gielen.2015,Noroozi.2014} for several examples of finite networks.
In particular, in case of exponential ISS we show both necessity and sufficiency of the proposed small-gain conditions.
We then truncate the infinite network to introduce a corresponding finite network which can be arbitrarily large and possibly \emph{unknown} in size.
We show that if each subsystem in the infinite network is individually ISS, all the stability and/or performance indices obtained for the infinite network are \emph{preserved} for its finite counterpart.
In that way, the small-gain conditions for the finite network are
\emph{independent} on the size of the network and hence, they can be used for
arbitrarily large networks, with possibly unknown size.
We illustrate the effectiveness of our results by application to an urban traffic network.

{\it Related literature:}
ISS theory of infinite-dimensional systems has been extensively studied in the recent years.
For an overview of this subject see, e.g.,~\cite{MiP20}.
However, most of the development is related to ISS of continuous-time systems,
with some exceptions, see, e.g.,~\cite[Section 9.6]{MiP20}, for references on ISS of infinite-dimensional impulsive systems combining continuous-time and discrete-time dynamics.

Recently, ISS small-gain theory for infinite networks of continuous-time systems has been intensively studied in~\cite{MKG20,DaP20,KMS19,NMK21}.
For \emph{small-gain theorems in trajectory formulation} decisive results have been obtained in \cite{MKG20}, where it was shown that a well-posed infinite network of input-to-state stable infinite-dimensional systems is again ISS, provided the so-called \emph{monotone limit property} holds. This property is slightly stronger than the \emph{uniform small-gain condition} which is also introduced in \cite{MKG20}, but is equivalent to the uniform small-gain condition in case of networks with a linear gain operator as well as for finite networks.
Furthermore, the results in \cite{MKG20} fully generalize available results for finite networks, provide spectral-type criteria for ISS of networks governed by linear and homogeneous gain operators, are applicable for a very broad classes of infinite networks and assume very mild regularity properties for subsystems and the whole network.

\emph{Lyapunov-based small-gain theorems} for continuous-time systems are now restricted to the case of countable interconnections of finite-dimensional components.
In~\cite{DaP20} it was shown that a network of an infinite number of ISS systems is also ISS, if all the nonlinear gains characterizing the influence of subsystems on each other, are uniformly less than identity.
By means of examples, it is established in~\cite{DMS19a} that classic max-form small-gain conditions (SGCs) developed for finite-dimensional systems~\cite{Dashkovskiy.2010} do not guarantee the stability of infinite networks of ISS systems, even if all the systems are linear.
To address this issue, more restrictive robust strong SGCs are developed in~\cite{DMS19a}, where local ISS Lyapunov functions are defined in implication form and the gain operator is used in a max formulation, which makes the gain operator nonlinear, even if all the gains are linear.
Nonlinearity of the gain operator makes the analysis of the infinite interconnection highly challenging, and the Lyapunov-based small-gain criteria obtained in \cite{DMS19a,DaP20} are not tight and more investigations are needed.

In contrast to this, for the case of linear gain operators tight results have been achieved in \cite{KMS19}, where sum-type SGCs for exponential ISS of infinite networks are developed.
 In~\cite{KMS19} each subsystem is assumed to be individually exponentially ISS and a small-gain condition in the form of a spectral radius criterion is presented. 
This work is further extended to exponential ISS with respect to closed sets in~\cite{NMK21}.
Such a generalized formulation of stability with respect to a closed set covers a wide range of stability/stabilization problems including incremental stability,  robust consensus/synchronization, ISS of time-varying systems as well as variants of input-to-output stability in a \emph{unified} setting~\cite{NMK21}.
In all of the above works, ISS of an interconnection is analyzed using a Lyapunov small-gain approach.

In classic ISS small-gain theorems, including all of the above
contributions, it is required that each subsystem is
\emph{individually} ISS to conclude ISS for the whole network.
This requirement is obviously conservative and rules out a large class of systems.
To address this conservatism, one may develop SGCs letting subsystems have stabilizing effect on each other.
To formulate such a setting for \emph{finite-dimensional} networks, the notion of finite-step Lyapunov functions~\cite{Geiselhart.2014c,Aeyels.1998} can be merged with Lyapunov-based small-gain methodology, which leads to so-called \emph{relaxed SGCs}~\cite{Geiselhart.2015,Gielen.2015,Noroozi.2014}.
A finite-step Lyapunov function is an energy function which does not have to decay every single time step, but only after a finite number of steps.
When this property of finite-step Lyapunov functions comes to a network of systems, it lets us look into future time steps of solutions and allow for considering potential stabilizing effect of subsystems on each other.
Interestingly enough, relaxed SGCs are shown to be both \emph{necessary} and \emph{sufficient} and therefore they can be applicable to networks with unstable subsystems.
In~\cite{Noroozi.2018a} relaxed SGCs for ISS of finite-dimensional networks with respect to closed sets have been developed and applications to incremental stability, ISS of time-varying networks and distributed observers design are discussed.

In this work, we extend the results in~\cite{Noroozi.2018a} to  infinite-dimensional systems.
Such a generalization leads to several \emph{nontrivialities} as this calls for a careful choice of an infinite-dimensional state space of the overall system,
and developing direct and converse Lyapunov theorems in an infinite-dimensional setting.
In terms of existing works in the infinite networks context, our work is close to~\cite{DaP20} as we also assume the internal ISS gains to be less than identity. However, our setting is formulated in discrete-time domain and that of~\cite{DaP20} is given in a continuous-time domain.
More importantly, the SGCs in~\cite{DaP20} are only sufficient, while here we show the \emph{necessity} of our formulation in case of exponential ISS.

A preliminary version of this work has been presented at~\cite{NMW20}.
In comparison to our previous work, we provide all proofs.
Additionally, in this work
we rigorously investigate under which conditions the system is well-defined
and well-posed.
In particular, we define the system on an extended state space with the structure of a Fr\'echet space, where the system is automatically
well-defined, cf. Section~\ref{sec:System description} below for more details.
While in~\cite{NMW20} we mainly assume the network to be well-posed, here we relate local stability conditions imposed on subsystems with well-posedness of the system and provide an illustrative example (cf. Lemma~\ref{lem:Uniform-K-boundedness}, Theorem~\ref{thm:Well-posedness-coupled-systems} and Example~\ref{examp:Example-well-posedness} below).
Finally, we discuss the preservation of ISS indices under truncation of an
infinite network in more generality than in~\cite{NMW20},
cf. Section~\ref{sec:From-infinite-to-finite-networks} below.

\section{Preliminaries}

\subsection{Notation}

We write $\N := \{1,2,3,\ldots\}$ for the set of positive integers and $\N_0 := \N\cup\{0\}$.
For vector norms on (in)finite-dimensional vector spaces, we write $|\cdot|$.
We use comparison functions formalism (see~\cite{Kellett.2014}):
{\allowdisplaybreaks
\begin{equation*}
\begin{array}{ll}
{\K} &:= \left\{\gamma:\R_+\rightarrow\R_+\left|\ \gamma\mbox{ is continuous, strictly} \right. \right. \\
&\phantom{aaaaaaaaaaaaaaaaaaa}\left. \mbox{ increasing and } \gamma(0)=0 \right\}, \\
{\K_{\infty}}&:=\left\{\gamma\in\K\left|\ \gamma\mbox{ is unbounded}\right.\right\},\\
{\LL}&:=\left\{\gamma:\R_+\rightarrow\R_+\left|\ \gamma\mbox{ is continuous and strictly}\right.\right.\\
&\phantom{aaaaaaaaaaaaaaaa} \text{decreasing with } \lim\limits_{t\rightarrow\infty}\gamma(t)=0\},\\
{\KL} &:= \left\{\beta:\R_+\times\R_+\rightarrow\R_+\left|\ \beta \mbox{ is continuous,}\right.\right.\\
&\phantom{aaaaaa}\left.\beta(\cdot,t)\in{\K},\ \beta(r,\cdot)\in {\LL},\ \forall t\geq 0,\ \forall r >0\right\}. \\
\end{array}
\end{equation*}
}
For $\alpha,\gamma\in\mathcal{K}$ we write $\alpha<\gamma$ if $\alpha(s)<\gamma(s)$ for all $s>0$.
Composition of functions is denoted by the symbol $\circ$ and repeated
composition of, e.g., a function $\gamma$ is defined inductively by
$\gamma^1:=\gamma$, $\gamma^{i+1}:=\gamma\circ \gamma^{i}$.

\subsection{System description}%
\label{sec:System description}

We study the interconnection of countably many systems, each
  given by a finite-dimensional difference equation.  The set $\N$ is used
  as the index set (by default). For each system $i$ we denote by $I_i
  \subset N\setminus \{ i \}$ the index sets of the neighbors of system $i$, i.e. the set
  of subsystems that directly influence the dynamics of system $i$. The
  $i$th subsystem is written as%
\begin{equation}\label{eq_ith_subsystem}
  \Sigma_i: \quad {x}_i^+ = f_i(x_i,\ol{x}_i,u_i).%
\end{equation}
Here $x_i \in \R^{n_i}$ is the state of the $i$th subsystem, $x_i^+$ denotes the state of the system $\Sigma_i$ at the next time step,
 $\ol{x}_i \in X(I_i)$ is the state of all the neighbors of system $x_i$, where 
\begin{eqnarray}
X({I_i}):= \prod_{j \in I_i} \R^{n_j}.
\label{eq:X(I_i)-space}
\end{eqnarray}

%
 
We impose the following assumptions on the subsystems $\Sigma_i$:
\begin{assumption}
\label{ass:Properties-of-Sigma_i}
The family $(\Sigma_i)_{i\in\N}$ comes together with sequences $(n_i)_{i\in\N}$, $(p_i)_{i\in\N}$ of positive integers and index sets $I_i \subset \N \backslash \{i\}$, $i\in\N$, so that the following holds
\begin{enumerate}[(i)]
\item\label{itm:gen-assumption-1}  The state vector $x_i$ of $\Sigma_i$ is an element of $\R^{n_i}$.%
\item\label{itm:gen-assumption-2}
The vector $\ol{x}_i =(x_j)_{j\in I_i}$ is composed of the state vectors 
$x_j$, $j \in I_i$.  
The space $X({I_i})$ is equipped with the norm $|\ol x_i|=|\ol x_i|_{ X({I_i})}:=\sup_{j\in I_i}|x_j|$.
For each $i\in \N$, the sets $I_i$ and $\{j\in \N: i\in I_j\}$ are
finite. 
\item\label{itm:gen-assumption-3} The external input vector $u_i$ is an element of $\R^{p_i}$.
\item\label{itm:gen-assumption-4} The right-hand side is a continuous function $f_i:\R^{n_i} \tm  X({I_i}) \tm \R^{p_i} \rightarrow \R^{n_i}$.
\end{enumerate}
\end{assumption}
 
In other words,
Assumption~\ref{ass:Properties-of-Sigma_i}\,\eqref{itm:gen-assumption-2}
requires that  each subsystem only has finitely many internal
inputs and provides an input only to finitely many other subsystems.
In system~\eqref{eq_ith_subsystem}, we consider $\ol{x}_i$ as an \emph{internal input} and $u_i$ as an \emph{external input}.

To define the overall network composed of subsystems $\Sigma_i$, we consider the state vector $x = (x_i)_{i\in\N}$, the input vector $u = (u_i)_{i\in\N}$ and the right-hand side 
\begin{eqnarray}
f(x,u) := (f_1(x_1,\ol{x}_1,u_1),f_2(x_2,\ol{x}_2,u_2),\ldots).
\label{eq:f-interconnection}
\end{eqnarray}
The overall system is then formally written as%
\begin{equation}\label{eq_interconnection}
  \Sigma:\quad x^+ = f(x,u).%
\end{equation}

Fix a norm on each $\R^{n_i}$ and define the state space $X$ for the system $\Sigma$ as
\begin{equation*}
 X:= \ell^{\infty}(\N,(n_i)) := \Bigl\{ x = (x_i)_{i\in\N} : x_i \in \R^{n_i},\ \sup_{i\in\N}|x_i| < \infty \Bigr\},%
\end{equation*}
and equip this space with the norm $|x|_{\infty} := \sup_{i\in\N}|x_i|$.
The space $\ell^{\infty}(\N,(n_i))$ is called the \emph{$l_\infty$-sum of spaces $\R^{n_i}$}, $i\in \N$, and it is a Banach space, see \cite[p.~127]{Hel06}. 

Note that by (ii) for all $x =(x_j)_{j\in \N}$, all $i\in\N$ and all $\ol{x}_i = (x_j)_{j\in I_i} \in  X({I_i})$ it holds that
\begin{equation}
|\ol{x}_i|\leq |x|_\infty.
\label{eq:norms-internal-input-state}
\end{equation}
Similarly, we consider the \emph{external input value space} $U := \ell^\infty(\N,(p_i))$, 
where we fix norms on $\R^{p_i}$ that we simply denote by $|\cdot|$ again. The norm on $U$ we denote by $|\cdot|_\infty$.
By the space $\UC$ of admissible \emph{external input functions}, we mean all sequences $u: \N_0 \rightarrow U$ 
such that 
$  \|u\|_{\infty} := \sup_{k\geq0}|u(k)|_\infty <\infty$.
The space $\UC$ with the norm $\|\cdot\|_{\infty}$ becomes a Banach space.

We also define  the \emph{extended state space} 
\begin{eqnarray}
X_E:= \prod_{i \in\N} \R^{n_i} .
\label{eq:Extended-state-space}
\end{eqnarray}
As $I_i$ is finite for any $i\in\N$, the state of any subsystem $\Sigma_i$ at any finite time $k\in\N$ is affected by only a finite number of other agents, and thus $f$ is always well-defined as a map from $X_E\tm U$ to $X_E $ and hence
for any $k\in\N$, any $u\in\UC$ and any $\xi \in X_E$ the corresponding solution to~\eqref{eq_interconnection}
 $x(k,\xi,u)$ is well-defined.

Thus, the equations \eqref{eq_interconnection} and space of input values $U$ define a forward complete control system in the extended space $X_E$, which we denote by $\Sigma=\Sigma(f,X_E,U)$ if we want to make the
data defining the system explicit.

The space $X_E$ has the structure of a Fr\'echet space, if endowed with
the countable family of seminorms $p_i: X_E \to \R_+$,
$p_i\left((x_j)_{j\in \N}\right) := |x_i|$, $i\in\N$. This yields a complete, locally
convex vector space with the  translation-invariant metric
\begin{equation*}
    d_E \big((x_i)_{i\in \N},(y_i)_{i\in \N} \big) = \sum_{i=1}^\infty
    \frac{1}{2^i} \frac{|x_i-y_i|}{1+|x_i-y_i|}.  
\end{equation*}
The induced topology is the topology of pointwise convergence, i.e.
for sequences $( x^k )_{k\in\N}= ( (x_i^k)_{i\in\N} )_{k\in\N}$, we have $x^k
\to x = (x_i)_{i\in\N}$ if and only if for all $i \in \N$ we have
\begin{equation*}
    \lim_{k \to \infty} x_i^k = x_i.
\end{equation*} Thus, if we
define the concept of ISS on $X_E$ in the usual manner, then we would
simply replace the norms by the distance to $0$ in $X_E$.

At the same time, for stability analysis we would like to ensure that the solutions of $\Sigma$ are well-defined in the space $X$.
This motivates
\begin{definition}
\label{def:well-posedness} 
We say that \emph{$\Sigma$ is well-posed}, if $f$ is well-defined as a map from $X \tm U$ to $X$. 
\end{definition}

Well-posedness ensures that for each initial condition $\xi \in X$ and for each input $u\in\cal{U}$ the solution of the system~\eqref{eq_interconnection} exists on $\N_0$ and lives in $X$, that is \eqref{eq_interconnection} is forward complete.
The following gives a complete characterization of well-posed systems on
$\ell^\infty(\N,(n_i))$.

\begin{lemma}\label{lem:well-posedness}
    Consider the systems $\Sigma_i$ as in \eqref{eq_ith_subsystem} and
    assume Assumption~\ref{ass:Properties-of-Sigma_i} holds.
		The following assertions are equivalent:
    \begin{enumerate}[(i)]
      \item\label{itm:well-posedness-1} The induced system $\Sigma$ is well-posed.
      \item\label{itm:well-posedness-2} $\Sigma$ is well-posed and there exist $C> 0$ and $\kappa\in\Kinf$ such that for all $x \in X$ and $u\in U$
		\begin{eqnarray}
		|f(x,u)|_\infty \leq C + \kappa(|x|_\infty)+ \kappa(|u|_\infty).
		\label{eq:uniform boundedness of f}
		\end{eqnarray}
      \item\label{itm:well-posedness-3} There exist $C> 0$ and $\kappa\in\Kinf$ such that
        for all $i \in \N$ and all $x_i\in \R^{n_i}, \ol{x}_i \in
         X({I_i}), u_i \in \R^{p_i}$ 
        \begin{equation}
        \label{eq:bound-lem:well-posedness}
            \vert f_i(x_i,\ol{x}_i, u_i)\vert \leq C +
            \kappa(\vert x_i\vert) +
            \kappa(\vert \ol{x}_i\vert) + \kappa(\vert u_i \vert).
        \end{equation}
    \end{enumerate}
\end{lemma}

\begin{proof}
\eqref{itm:well-posedness-3} $\Rightarrow$ \eqref{itm:well-posedness-2}: Let $(x,u) \in X \times U$. Using \eqref{eq:norms-internal-input-state}, we have for each
$i$ that 
\begin{multline*}
    \vert f_i(x_i,\ol{x}_i, u_i)\vert \leq C +
            \kappa(\vert x_i\vert) +
            \kappa(\vert \ol{x}_i\vert) + \kappa(\vert u_i \vert)
\\ \leq C +
            \kappa(\vert x\vert_\infty) +
            \kappa(\vert x\vert_\infty) + \kappa(\vert u \vert_\infty).
\end{multline*}
As the right hand side is independent of $i$ this shows that
$(f_i(x_i,\ol{x}_i, u_i))_{i\in \N} \in X$ and so $f(x,u) \in X$.

\eqref{itm:well-posedness-2} $\Rightarrow$ \eqref{itm:well-posedness-1}: Clear.

\eqref{itm:well-posedness-1} $\Rightarrow$ \eqref{itm:well-posedness-3}: By assumption, each $f_i$ is continuous, so that
with $C_i := \vert f_i(0,0, 0)\vert$
and $\kappa_i \in \K$ defined by
\begin{equation*}
    \kappa_i(r) := \max \left \{\vert f_i(x_i,\ol{x}_i, u_i) -
      f_i(0,0, 0) \vert  \ ; \  \vert x_i\vert, \vert
      \ol{x}_i\vert,\vert u_i \vert \leq r \right \},
\end{equation*}
we have for all $i\in \N$ that always
\begin{equation}
\label{eq:i-bound}
    \vert f_i(x_i,\ol{x}_i, u_i)\vert \leq C_i +
            \kappa_i (\vert x_i\vert) + \kappa_i(\vert \ol{x}_i\vert) + \kappa_i(\vert u_i \vert).
\end{equation}
If $\sup_{i \in \N} C_i$ is not finite, then $f$ is not well posed, as
then $f(0,0) \notin X$.  
Otherwise, we set $C:= \sup_{i \in \N} C_i$ and 
define $\kappa: \R_+ \to \R_+ \cup \{ \infty \}$ by
\begin{equation*}
    \kappa(r) := \sup_{i\in \N} \kappa_i(r),\quad r\geq 0. 
\end{equation*}
As the supremum of
continuous increasing functions, $\kappa$ is lower semicontinuous and
nondecreasing on its domain of definition.

If $\kappa(r)$ is finite for every $r\geq 0$, define 
\[
\tilde{\kappa}(r):=
\begin{cases}
0, & r=0,\\
\kappa(r)-a, & r>0,
\end{cases}
\]
where $a:=\lim_{r\to+0}\kappa(r)\geq 0$ (the limit exists as $\kappa$ is nondecreasing). 
By construction, $\tilde{\kappa}$ is nondecreasing, continuous at $0$ and $\tilde{\kappa}(0)=0$.
It is easy to see that $\tilde{\kappa}$ can be upper bounded by a certain $\hat{\kappa}\in {\cal K}_\infty$ (this follows from a more general result in \cite[Proposition 9]{MiW19b}). 

Then $\kappa(r) \leq \hat{\kappa}(r) + a$ for all $r\geq 0$ and we obtain from \eqref{eq:i-bound} that 
\begin{equation}
\label{eq:i-bound-2}
    \vert f_i(x_i,\ol{x}_i, u_i)\vert \leq C + 3a +
            \hat{\kappa} (\vert x_i\vert) + \hat{\kappa}(\vert \ol{x}_i\vert) + \hat{\kappa}(\vert u_i \vert),
\end{equation}
which shows the claim.

It
thus remains to show that if $\kappa(R) = \infty$ for some $R >0$, then
$\Sigma$ is not well-posed. If $\kappa(R) = \infty$, there exist sequences 
\begin{equation*}
    ( i_k )_{k\in\N}, ( x_{i_k} )_{k\in\N}, (\ol{x}_{i_k})_{k\in\N}, ( u_{i_k})_{k\in\N}   
\end{equation*}
such that for all $k$: $i_k \in \N$, $\vert x_{i_k}\vert,  \vert
\ol{x}_{i_k}\vert,  \vert u_{i_k} \vert \leq R$ and such that
\begin{equation}
\label{eq:unboundedsec}
    \vert f_{i_k}(x_{i_k},\ol{x}_{i_k}, u_{i_k})\vert \geq k.
\end{equation}
By assumption, the neighborhood graph of the network is locally finite. In
particular, for any $i \in \N$, the set 
\begin{equation*}
    \{ j \in \N \ ; \ I_j \cap I_i \neq \emptyset \} 
\end{equation*}
is finite. We may thus assume, by taking a subsequence if necessary, that
for all indices $k,\ell\in \N$ we have 
\begin{equation}
\label{eq:emptyintersec}
  \left(\{i_k \} \cup I_{i_k} \right)
\cap \left( \{ i_\ell \} \cup I_{i_\ell} \right) = \emptyset  .
\end{equation}
Otherwise,
by local finiteness, a certain index $i^* \in \N$ would appear infinitely often in the
sequence $\{ i_k \}$, but this contradicts the bound in \eqref{eq:i-bound}
applied to $i^*$ and the bound \eqref{eq:unboundedsec}. We may
now define points $(x,u)= ((x_i),(u_i)) \in X \times U$ by setting
\begin{align*}
   x_i &:= \left \{ 
   \begin{matrix}
       x_{i_k} & \quad & \text{if } i=i_k \text{ for some } k\in \N, \\
       x_j    &       &  \text{if } i=j \in I_{i_k}  \text{ for some }
       k\in \N,  \\ 
       0 & & \text {otherwise;} 
   \end{matrix}\right. \\ 
   u_i &:= \left \{ 
   \begin{matrix}
       u_{i_k} & \quad & \text{if } i=i_k \text{ for some } k\in \N, \\
        0 & & \text {otherwise} 
   \end{matrix}\right..
\end{align*}
This is well-defined by \eqref{eq:emptyintersec} and $(x,u) \in X\times U$
by the choice of the sequences. Finally, $f(x,u) \notin X$ by \eqref{eq:unboundedsec}. This
completes the proof.
\end{proof}

\begin{remark}
\label{rem:BRS} 
We say that a well-posed system $\Sigma$ has bounded reachability sets (BRS) if for all bounded balls $B_1$ in $X$ and $B_2$ in $\UC$ and for any time $k$ there is a bounded ball $B_3 \subset X$ such that $x(s,x_0,u)\in B_3$ for all $x_0 \in B_1$, $u\in B_2$, $s=1,\ldots,k$. 

Using a discrete-time variation of the characterization of BRS in \cite[Lemma 3]{MiW18b}, it is easy to see that condition~\eqref{itm:well-posedness-2} is equivalent to BRS for $\Sigma$, and thus the equivalence 
\eqref{itm:well-posedness-1} $\Iff$ \eqref{itm:well-posedness-2} in Lemma~\ref{lem:well-posedness} reads as well-posedness is equivalent to BRS.
\end{remark}

\subsection{Distances in sequence spaces}%

We continue to assume that $X=\ell^\infty(\N,(n_i))$. Consider nonempty closed sets $\A_i \subset \R^{n_i}$, $i\in\N$. For each $x_i \in \R^{n_i}$ we define the distance of $x_i$ to the set $\A_i$ by%
\begin{equation*}
  |x_i|_{\A_i} := \inf_{y_i\in \A_i}|x_i - y_i|.
\end{equation*}

Let $\A(I_i) := \prod_{j\in I_i} \A_j$, and note that for $\ol{x}_i =
(x_j)_{j\in I_i} \in X({I_i})$ it holds that
\begin{align}
\label{eq:Distance-in-overline-A_i}
  |\ol{x}_i|_{ \A(I_i)} 
	:=& \inf_{z\in  \A(I_i)}|\ol{x}_i - z| 
	= \inf_{(z_j)_{j\in I_i} \in \A(I_i)  }\max_{j\in I_i}|x_j - z_j| \nonumber\\
	=& \max_{j\in I_i}|x_j|_{\A_j},	
\end{align}
where, in the last equality, we have used that the choice of $z_j\in
\A_j$, $j\in I_i$ is independent of the other indices.

Further define the set%
\begin{align}\label{eq:Overall_set}
\hspace{-3mm}  \A := \{x \in X: x_i \in \A_i,\ i\in\N\} {=}  X \cap \prod_{i\in\N}\A_i.%
\end{align}

\begin{definition}
\label{def:uniformly bounded-A_i} 
We say that $(\A_i)_{i\in\N}$ is \emph{uniformly bounded}, if there is a
$C>0$ such that  for all $i\in\N$ and all $x\in \A_i$ it holds that $|x|\leq C$. 
\end{definition}
Note that if $(\A_i)_{i\in\N}$ is uniformly bounded, then $\A = \prod_{i\in\N}\A_i $.
If $\A \neq \emptyset$, we define the distance from $x\in X$ to $\A$ as%
\begin{eqnarray}\label{eq:Distance_to_overall_set}
  |x|_{\A} := \inf_{y\in \A}|x-y|_\infty = \inf_{y\in \A} \sup_{i\in\N} |x_i-y_i| .%
\end{eqnarray}

As $X$ is infinite dimensional, it is in general not true for any
closed set $K$ that for $x\in X$ there is a minimizing point $y \in K$
with $|x-y|_\infty=|x|_{K}$. The following is thus a remarkable property
of the set ${\cal A}$. The proof uses the well-known property that in
$\R^n$ the distance to a closed set has a minimizer. We include a proof
for the convenience of the reader.

\begin{lemma}\label{lem:Alternative-A-representation} 
Let $X=\ell^\infty(\N,(n_i))$.
Assume that $\A$ defined by \eqref{eq:Overall_set} is nonempty. Then for any $x \in X$ there is some $y^* \in \A$ such that
\begin{align}\label{eq:Distance_to_overall_set_Formula}
  |x|_{\A} = \sup_{i\in\N} |x_i|_{\A_i} = |x - y^*|_\infty.
\end{align}
\end{lemma}

\begin{proof}
Fix $x\in X$ and $y \in\A$. Then
\begin{align}
\abs{x-y} = \sup_{i\in\N} \abs{x_i- y_i} \geq  \sup_{i\in\N} \abs{x_i}_{\A_i} ,
\end{align}
as each $y_i \in\A_i$.
Taking the infimum over $y \in \mathcal{A}$, we get
\begin{align} \label{eq:distance-proof-1}
|x|_{\A} = \inf_{y \in \A} |x - y|_\infty \geq \sup_{i \in \N} |x_i|_{\mathcal{A}_i} .
\end{align}

Conversely, pick any $x\in X$ and for each $i\in\N$ consider the set
\begin{eqnarray}
B_i:= \A_i \cap \{z \in\R^{n_i}: |z-x_i|\leq \abs{x_i}_{\A_i} + 1 \},
\label{eq:Auxiliary-set}
\end{eqnarray}
which is a non-empty, closed, bounded, and hence compact subset of $\A_i$, on which $y_i \mapsto |x_i - y_i|$ is a continuous function.


Hence, there is at least one minimizer $y^*_i\in\A_i$ of $\inf_{y_i \in \A_i} |x_i - y_i|$, which also belongs to $B_i$.

As $\A\neq\emptyset$, there is $R>0$, such that for all $i\in\N$ there is $z_i\in\A_i$ with $\abs{z_i}\leq R$.
Hence, for any $i\in\N$ and any $x\in X$ we have that 
\[
\abs{x_i}_{\A_i} = \inf_{z\in \A_i}|x_i-z|\leq |x_i-z_i|\leq |x_i| + |z_i| \leq |x|_\infty+R,
\]
and by \eqref{eq:Auxiliary-set}, $|y^*_i|\leq|x|_\infty+R+1$.

Thus, for a given $x \in X$, there exists some $y^*=(y^*_i)_{i\in\N} \in \mathcal{A}$ such that, for all $i \in \N$,
$$
|x_i - y^*_i| = |x_i|_{\mathcal{A}_i}.
$$
But then,
$$\sup_{i \in \N} |x_i|_{\A_i} = \sup_{i \in \N} |x_i - y^*_i| = |x - y^*|_\infty \geq |x|_\A,$$
as required.
\end{proof}

We stress that the statement is
only applicable for $\A \subset X =\ell^\infty(\N,(n_i))$, while in
general  $\prod_{i\in\N}\A_i$ may contain only unbounded sequences.
Also note that if $\A=\{0\}$, then $|x|_{\{0\}} = |x|_\infty$.%


\section{Input-to-State Stability}\label{sec:ISS}

We aim to establish the stability of the interconnected system with respect to a closed set $\A \subset X$ via so-called finite-step Lyapunov functions.
For this purpose, we introduce the notions of input-to-state stability (ISS)  and of $\K$-boundedness with respect to $\A$.
Throughout this section we assume:
\begin{assumption}\label{ass:Monotonicity-of-norm}
$\Sigma$ is well-posed.
\end{assumption}
 
\begin{definition}\label{def:K-boundedness}
Let a nonempty closed set $\A \subset X $ be given.
A function $f:X \times U \to X$  is called \emph{$\mathcal{K}$-bounded with respect to $\A$}, if there are $\kappa_1,\kappa_2\in \KC$ such that
for all $\xi\in X$ and $\mu\in U$
\begin{align}\label{eq:K-bound-2}
\abs {f(\xi, \mu)}_\A \le \kappa_1 (\abs {\xi}_\A )+\kappa_2 (|\mu|_{\infty}).
\end{align}
\end{definition}

\begin{remark}
In view of item \eqref{itm:well-posedness-2} of Lemma~\ref{lem:well-posedness}, it is not hard to see that in case of $\A=\{0\}$, 
$\K$-boundedness of $f$ is equivalent (under assumptions on $\Sigma$ which we have imposed) to well-posedness of $f$, its continuity at $0$ and $f(0,0)=0$.
\end{remark}

\begin{definition}\label{def:eISS}
Given a nonempty closed set $\A \subset X$, system $\Sigma$ is said to be \emph{input-to-state stable (ISS) with respect to $\A$} if there exist $\beta\in \KL$ and $\gamma \in \KC$ such that for any initial state $\xi \in X$ and any input $u\in\UC$ the corresponding solution to~\eqref{eq_interconnection} satisfies%
\begin{equation}\label{eq:ISS}
  \hspace{-1mm}|x(k,\xi,u)|_\A \leq \max\big\{ \beta( |\xi|_\A,k) , \gamma(\|u\|_{\infty}) \big\} \,\,\forall k \in \N_0 .%
\end{equation}
$\Sigma$ is called \emph{exponentially input-to-state stable (eISS) with respect to $\A$}, if there are constants $C\geq 1, \rho \in [0,1)$ and $\gamma \in \KC$ such that $\beta$ in~\eqref{eq:ISS} can be chosen to be $C\rho^k |\cdot|_{\mathcal A}$.
\end{definition}

As observed in~\cite{GeW16,Noroozi.2018a}, for finite-dimensional systems every system which is ISS with respect to $\A$ is \emph{necessarily} $\K$-bounded with respect to $\A$.
Here we show that this observation also holds for infinite network~\eqref{eq_interconnection}, see Proposition~\ref{prop:ISS-criterion} below and the corresponding set $\A$ is invariant for the dynamics of $\Sigma$ if $u\equiv 0$.

The underlying idea is to formulate stability properties from a subsequence of state trajectories.
In particular, we aim to understand, whether one can conclude ISS
 by only looking at solutions every $M$ time steps, with $M\in\N$.

We define the iterates of $f$ inductively, by $f^1:=f$, and
$f^{k+1}:X_E\times U^{k+1}\to X_E$ by  
$f^{k+1}(x,(u_0,\ldots,u_{k}))  = f\big(f^{k}\big(x,(u_0,\ldots,u_{k-1})\big),u_k\big)$. The
$M$-iterate system is then $\Sigma^M := \Sigma(f^M,X_E,U^M)$.

\begin{remark}
\label{rem:Well-posedness-of-iterates} 
If $\Sigma$ is well-posed, that is $f$ is well-defined as a map from $X\tm U$ to $X$, then also $\Sigma^M$ is well-posed for any $M>1$.
However, it may be that $\Sigma^M$ is well-posed for a certain $M>1$, but $\Sigma$ is not well-posed.
E.g. in Section~\ref{sec:eISS-infinite-networks}
 we verify well-posedness of $\Sigma^M$ without having well-posedness of $\Sigma$ in advance.
\end{remark}

The following lemma establishes that $\Sigma^M$ is ISS if and only if the trajectories of $\Sigma$ satisfy an ISS-like estimate every $M$ time steps.
\begin{lemma}
\label{lem:M-iterate-eISS}
Given a nonempty closed set $\A \subset X$ and $M \in \N$, $\Sigma^M$ is ISS with respect to $\A$, if and only if there are $\beta\in \KL$ and $\gamma \in \KC$ such that for any initial state $\xi \in X$ and any $u\in\UC$ the corresponding solution to~\eqref{eq_interconnection} satisfies for all $k \in \N_0$
\begin{equation}\label{eq:ISS-M}
 |x(Mk,\xi,u)|_\A \leq \max\{ \beta( |\xi|_\A,k) , \gamma(\|u\|_{\infty}) \}.
\end{equation}
\end{lemma}

The following proposition shows a useful criterion for ISS of discrete-time systems. 
\begin{proposition}
\label{prop:ISS-criterion} 
Let $\A \subset X$ be closed.
\begin{enumerate}[(i)]
	\item $\Sigma$ is ISS with respect to $\A$ $\qiq$ there is $M\in\N$ such that $\Sigma^M$ is ISS with respect to $\A$, and $f$ is $\mathcal{K}$-bounded with respect to $\A$.
	\item $\Sigma$ is eISS with respect to $\A$ $\qiq$ there is $M\in\N$ such that $\Sigma^M$ is eISS with respect to $\A$, and $f$ is $\mathcal{K}$-bounded with respect to $\A$ with a linear $\kappa_1$.
\end{enumerate}
\end{proposition}

\begin{proof}
(ii) is analogous to (i), thus we show only (i).

$\Rightarrow$.
Clearly, ISS implies ISS of the $M$-iterate of $\Sigma$ for any $M\in\N$.

Furthermore, from~\eqref{eq:ISS} and the causality of $\Sigma$ we have
$$
\abs{f(\xi,u)}_\A = \abs{x(1,\xi,u)}_\A \leq \beta\big(\abs{\xi}_\A,1\big)+ \gamma(\abs{u}_\infty) ,
$$
which implies global $\K$-boundedness with $\kappa_1 (\cdot) =  \beta\big(\cdot,1\big)$ and $\kappa_2 (\cdot) = \gamma(\cdot)$.

$\Leftarrow$. Let $M\in\N$ be such that $\Sigma^M$ is ISS with respect to $\A$.
Pick any $\xi\in X$, $u\in\UC$ and any $t \in\N$. For this $t$ there are unique $k\in\N$ and $j\in\{0,1,\ldots,M-1\}$ such that
$t=kM+j$.

By the cocycle property of solutions we have that 
\begin{align*}
  |x(t,&\xi,u)|_\A = |x(Mk+j,\xi,u)|_\A \\
									=& |x(1, x(Mk+j-1,\xi,u),u(\cdot+Mk+j-1))|_\A.
\end{align*}
As $\Sigma$ is $\K$-bounded, we obtain that 
\begin{eqnarray*}
  |x(t,\xi,u)|_\A \leq \kappa_1(|x(Mk+j-1,\xi,u)|_\A) + \kappa_2(\|u\|_\infty).
\end{eqnarray*}
Using $\K$-boundedness again, and exploiting the trivial inequality $\kappa(a+b)\leq \kappa(2a)+\kappa(2b)$, which is valid for any $a,b\geq 0$, we obtain that
\begin{align*}
  |x(t,&\xi,u)|_\A \\
	\leq& \kappa_1\big(\kappa_1(|x(Mk{+}j{-}2,\xi,u)|_\A) + \kappa_2(\|u\|_\infty)\big)\\
	& + \kappa_2(\|u\|_\infty)\\
  \leq& \kappa_1\big(2\kappa_1(|x(Mk+j-2,\xi,u)|_\A)\big) + \kappa_1\big(2\kappa_2(\|u\|_\infty)\big)\\
  &+ \kappa_2(\|u\|_\infty).
\end{align*}
By induction, there exist $\rho,\eta \in\Kinf$, which are independent of $j\in\{0,\ldots,M-1\}$, such that 
\begin{eqnarray*}
  |x(t,\xi,u)|_\A \leq \rho(|x(Mk,\xi,u)|_\A) + \chi(\|u\|_\infty).
\end{eqnarray*}
As $\Sigma^M$ is ISS, we proceed to 
\begin{align*}
  |x(t,\xi&,u)|_\A \leq \rho(\beta(|\xi|_\A,k) + \gamma(\|u\|_\infty)) + \chi(\|u\|_\infty)\\
								 \leq &\rho(2\beta(|\xi|_\A,k)) + \rho(2\gamma(\|u\|_\infty)) + \chi(\|u\|_\infty)\\
								 \leq &\max\!\big\{2\rho(2\beta(|\xi|_\A,k)) , \rho(2\gamma(\|u\|_\infty)) \!+\! \chi(\|u\|_\infty)\big\},
\end{align*}
which shows $\Sigma$ is ISS.

With the same steps as above and considering the linearity of $\kappa_1$ one can show that 
if $\Sigma^M$ is eISS then $\Sigma$ is eISS.
\end{proof}

\begin{remark}
Proposition~\ref{prop:ISS-criterion} has been shown for finite-dimensional systems in \cite[Remark 4.2, Corollary 4.3]{GeW16} on the basis of converse ISS Lyapunov theorems for finite-dimensional discrete-time systems. 
As such converse ISS Lyapunov results are not available for infinite-dimensional systems, we gave a direct non-Lyapunov proof for this fact.
\end{remark}

Now we pursue our strategy by introducing \emph{finite-step} ISS Lyapunov functions.

\begin{definition}\label{def:eISS-Lyapunov-Function}
Let a nonempty closed set $\A \subset X $ be given. 
A continuous function $V:X \rightarrow \R_+$ is called a \emph{finite-step ISS Lyapunov function} for $\Sigma$ with respect to~$\A$ if there exist $M \in\N$, $\ul\omega,\ol\omega, \alpha\in \Kinf$ with $\alpha < \id$ and $\gamma \in \K$ such that%
\begin{subequations}\label{eq_iss_lyap_props}
\begin{align}
\ul\omega (|\xi|_\A) &\leq V(\xi) \leq \ol\omega (|\xi|_\A) , \label{eq_iss_lyap_props-1}\\
 V (x(M,\xi,u))&\leq \max\big\{ \alpha( V(\xi)) , \gamma(\|u\|_{\infty})\big\} , \label{eq_iss_lyap_props-2}%
\end{align}
\end{subequations}
hold for all $\xi \in X$ and $u \in \UC$.

The function $V$ is called a finite-step \emph{eISS} Lyapunov function for $\Sigma$ with respect to~$\A$ if there are $M \in\N$, constants $\ul{w},\ol{w},b> 0,$ $\kappa\in[0,1)$ and $\gamma \in \K$ such that%
\begin{subequations}\label{eq_eiss_lyap_props}
\begin{align}
\ul{w} |\xi|^b_\A &\leq V(\xi) \leq \ol{w} |\xi|^b_\A , \label{eq_eiss_lyap_props-1}\\
 V (x(M,\xi,u))&\leq \max\big\{\kappa V(\xi) , \gamma(\|u\|_{\infty})\big\} , \label{eq_eiss_lyap_props-2}%
\end{align}
\end{subequations}
hold for all $\xi \in X$ and $u \in \UC$.
If inequality~\eqref{eq_iss_lyap_props-2} (resp.~\eqref{eq_eiss_lyap_props-2}) holds with $M=1$, then we drop the term ``finite-step'' and simply speak of an ISS (resp. eISS) Lyapunov function.
\end{definition}

Lyapunov functions, for which there exists $\ul\omega\in\KC_\infty$ so that $V(x) \geq \ul\omega(|\xi|_\A)$ for all $\xi \in X$, are called \emph{coercive} Lyapunov functions~\cite{MiW18b}.

Note that every ISS Lyapunov function is necessarily a finite-step Lyapunov function for any $M$.
However, a finite-step Lyapunov function does \emph{not} have to decay every single time step, but only every $M$ time steps.
Such a relaxation is useful in analysis and design of control systems; see e.g.~\cite{Geiselhart.2014c,Gielen.2015,Geiselhart.2015,Noroozi.2018a,Noroozi.2020a} for finite-step Lyapunov function-based analysis and synthesis of  \emph{finite}-dimensional control systems.
In particular, as far as large-scale networks are concerned, the use of finite-step Lyapunov functions enables us to introduce small-gain conditions which are not only sufficient, but also necessary for the verification of ISS of the network (see Theorem~\ref{nec} below).

Now we show that the existence of a finite-step ISS Lyapunov function guarantees ISS of the system.

\begin{proposition}\label{prop_iss}
Consider a system $\Sigma(f,X,U)$ with $\K$-bounded $f:X\times U \to X$.
If there exists a finite-step ISS Lyapunov function for $\Sigma$ with respect to $\A$, then $\Sigma$ is ISS with respect to $\A$.
Additionally, if $\kappa_1$ in~\eqref{eq:K-bound-2} is a linear function,
the existence of a finite-step eISS Lyapunov function for $\Sigma$
with respect to $\A$ implies eISS of $\Sigma$ with respect to $\A$.
\end{proposition}

\begin{proof}
The proof boils down to the equivalence between ISS of $\Sigma$ and ISS of $\Sigma^M$.
A finite-step ISS Lyapunov function (with a given $M$) is a 1-step ISS Lyapunov function for $\Sigma^M$, which by classic direct ISS Lyapunov Theorem\footnote{The proof of the direct ISS Lyapunov theorem for infinite-dimensional systems follows similar arguments as those given in~\cite[Lemma 3.5]{Jiang.2001} and~\cite[Theorem 7]{Noroozi.2018a} for finite-dimensional systems. Therefore, it is not presented here.} implies ISS of
$\Sigma^M$. Proposition~\ref{prop:ISS-criterion} finishes the proof.
\end{proof}

Here we present a \emph{converse} finite-step eISS Lyapunov theorem which gives us an explicit formula to compute the integer $M$ in~\eqref{eq_eiss_lyap_props-2}.

\begin{proposition}\label{thm:converse-lyap}
Let $\A \subset X $ be nonempty and closed.
Suppose that  system $\Sigma$ is eISS with respect to $\A$ with $\rho \in [0,1)$, $C\geq 1$ and $\gamma\in\K$. Then for any function $V:X \rightarrow \R_+$ and constants $\ul{w},\ol{w},b > 0$ satisfying
\begin{align}
\ul w |\xi|^b_\A &\leq V(\xi) \leq \ol w |\xi|^b_\A ,\, \forall \xi \in X  \label{eq_eiss_lyap_props-3}
\end{align}
and for all $\kappa \in (0,1)$ there exists an $M \in \N$ such that 
\begin{align}
\hspace{-3mm} V (x(M,\xi,u))&\leq \max\big\{ \kappa V(\xi) , \ol w \gamma(\|u\|_{\infty})^b \big\} \,\,, \label{eq_eiss_lyap_props-4}%
\end{align}
for all $\xi \in X$ and all $u\in\UC$.
In particular, one can choose any $M$ satisfying $M \geq \frac{1}{b} \log_\rho (\frac{\kappa \ul w }{C^b\ol w })$,
  and $V$ is an exponential $M$-step ISS Lyapunov function for $\Sigma$.
	\end{proposition}
	\begin{proof}
Consider any function $V:X \to \R_+$ satisfying~\eqref{eq_eiss_lyap_props-3} and let $\Sigma$ be eISS with a certain $\rho\in (0,1)$.
For every $M \in \N$ such that $M \geq \frac{1}{b} \log_\rho (\frac{\kappa \ul w }{C^b\ol w })$ it follows that $\ol w C^b \rho^{Mb}\le \kappa \ul w$.
The use of the eISS property of $\Sigma$,~\eqref{eq_eiss_lyap_props-3} and the above choice of $M$ give
\begin{align}
V(x(M,\xi,u)) & \leq \ol{w} |x(M,\xi,u)|^b_\A \nonumber\\
&\leq \ol{w} \max \big\{ C^b \rho^{Mb}|\xi|^b_\A , \gamma(\|u\|_{\infty})^b \big\} \nonumber\\ 
&\leq \max \big\{ \ul{w}\kappa |\xi|^b_\A , \ol{w} \gamma(\|u\|_{\infty})^b \big\} \nonumber\\
& \leq \max \big\{\kappa V(\xi) , \ol{w}\gamma(\|u\|_{\infty})^b \big\} ,  \label{eq_eiss_lyap_props-5}
\end{align}
for all $\xi \in X$. Thus~\eqref{eq_eiss_lyap_props-4} holds  with any $M \geq \frac{1}{b} \log_\rho (\frac{\kappa \ul w }{C^b\ol w })$, which
completes the proof.
\end{proof}
The above converse Lyapunov result provides an explicit formulation for a class of finite-step Lyapunov functions, which can be used for control purposes.
Given the desired exponential decay rate of solutions for a closed-loop system and a function $V$ satisfying~\eqref{eq_eiss_lyap_props-3}, one can immediately compute the corresponding positive integer $M$ from Proposition~\ref{thm:converse-lyap}.
This proposition is an infinite-dimensional extension of Theorem 10 in~\cite{Noroozi.2020a} which has been already used for controller design, see~\cite[Proposition 14, Theorem 23]{Noroozi.2020a} for more details.

\section{ISS for Infinite Networks}
\label{sec:eISS-infinite-networks}

The main objective of this work is to develop conditions for input-to-state stability of the interconnection of countably many subsystems \eqref{eq_ith_subsystem}, depending on certain stability-like properties of the subsystems.%

As argued in Section~\ref{sec:System description},
the solution of the overall system is well-defined in the extended state space, i.e. $x(k,\xi,u) \in X_E$ for all $k\in\N$, $\xi \in X_E$ and $u\in\UC$. At the same time, for now we do not require that $\Sigma$ is well-posed in the sense of Definition~\ref{def:well-posedness}.

We assume that each subsystem~\eqref{eq_ith_subsystem} satisfies a finite-step ISS Lyapunov-like condition, which is formulated by the following assumption.

\begin{assumption}\label{ass_vi_existence}
Given $M \in \N$ and nonempty closed sets $\A_i \subset \R^{n_i}$, $i\in\N$, for each subsystem $\Sigma_i$, $i\in\N$ there exists a continuous function $W_i:\R^{n_i} \rightarrow \R_+$ with the following properties.
\begin{itemize}
\item There are $\ul{\omega}_i,\ol{\omega}_i \in \Kinf $ so that for all $\xi_i \in \R^{n_i}$%
\begin{equation}\label{eq_viest}
  \ul{\omega}_i (|\xi_i|_{\A_i}) \leq W_i(\xi_i) \leq \ol{\omega}_i (|\xi_i|_{\A_i}) .%
\end{equation}
\item There are $\gamma_{ij} \in  \Kinf \cup \{0\}$ and $\gamma_{iu} \in \K$ so that for all $\xi \in X$, all $u \in\UC$  the following holds 
\begin{align}\label{eq_nablaviest}	
\!\!\!\!\!\! W_i (x_i(M,\xi,u)) \!\leq \!\max \!\big\{\! \sup_{j \in \N} \gamma_{ij} (W_j(\xi_j)) , \gamma_{iu} (\|u\|_{\infty}) \big\}, 
\end{align}
where $x_i(M,\xi,u)$ denotes the $i$th component of $x(M,\xi,u) \in X_E$.
\end{itemize}
\end{assumption}

In the following, we denote by $I_i(M)$ the index set of all
  subsystems, which influence $\Sigma_i$ on the interval $[0,M]$. In other
words, the set of all indices $j\in \N$ such that the knowledge of the
initial condition of $\Sigma_j$ is required to determine the component
$x_i(M)$ of the solution of $\Sigma$.

Note that $x_i (M)$ on the left hand side of~\eqref{eq_nablaviest} is the $i$th component of the solution $x$ to the network.
However, one does not need to know the entire dynamics to
compute~\eqref{eq_nablaviest} for the $i$th subsystem. For the sake of discussion, let $u = 0$.  If $M =
1$, then the computation of $x(1)$ only
requires the knowledge of $f_i$, $\xi_i$ and $\ol \xi_i$.  Now if $M =2$, then we require not
only the information from the neighbors, but also their neighbors,
i.e. $\ol \xi_i$ and $\ol \xi_j$ with $j\in I_i$ and $f_j$ with $j\in I_i$.  Similar arguments hold for $M
\geq 3$.
This justifies, that the supremum in~\eqref{eq_nablaviest}	is taken only over the set $j\in I_i(M) \cup \{i\}$.

The existence of $W_i$ in Assumption~\ref{ass_vi_existence} does \emph{not} imply that subsystem $\Sigma_i$ is individually ISS, i.e. $W_i$ \emph{cannot} be viewed as a finite-step Lyapunov function of subsystem $\Sigma_i$.
In other words, subsystem $\Sigma_i$ can be individually unstable, though there exists a function $W_i$ satisfying Assumption~\ref{ass_vi_existence}.
This comes from the fact that subsystems can have stabilizing effect on each other and an \emph{individually} unstable subsystem can be stabilized through receiving stabilizing effect from its neighbors; e.g. see illustrative examples in~\cite[Section V]{Geiselhart.2015} for more details.
The stabilizing effect is observed by looking at solutions of the system in few steps ahead, which is captured by~\eqref{eq_nablaviest}.
This in contrast with \emph{classic} small-gain theorems where potential stabilizing effects are \emph{not} taken into account, and hence every single subsystem has to be individually ISS.

To state the small-gain theorem, we also make the following \emph{uniformity} condition for the functions introduced in 
Assumption~\ref{ass_vi_existence}.

\begin{assumption}\label{ass_external_gains}
Assume that for subsystem $\Sigma_i$, $i \in\N$, the following hold.
\begin{enumerate}[(i)]
\item There exist $\ul{\omega},\ol{\omega}\in\Kinf$ so that for all $i\in \N$
\begin{equation}\label{eq_uniformity_alpha}
  \ul{\omega} \leq \ul{\omega}_i \leq \ol{\omega}_i \leq \ol{\omega} .%
\end{equation}
\item 
There exists $\alpha\in\Kinf$ with $\alpha <\id$ so that for all $i,j\in\N$
\begin{equation}\label{eq_uniformity_gamma}
   \gamma_{ij} \leq \alpha . 
\end{equation}
\item There is $\ol{\gamma}_u \in \K$ so that for all $i\in\N$%
\begin{equation}\label{eq_uniformity_gammaiu}
  \gamma_{iu} \leq \ol{\gamma}_u.
\end{equation}
\end{enumerate}
\end{assumption}

Conditions~\eqref{eq_uniformity_alpha} and~\eqref{eq_uniformity_gamma}  are necessary to construct an overall finite-step
Lyapunov function which is well-defined and coercive.
Condition~\eqref{eq_uniformity_gamma} rules out nonuniform decay rates for the solutions of the subsystems.
Without this uniform asymptotic stability of the interconnection need not hold, even if the system is linear and all internal and external gains are zero.
To see this, consider the following example
\begin{eqnarray*}
{x}_i^+ = \frac{i}{i+1} x_i + u,\quad i\in\N,%
\end{eqnarray*}
where $x_i,u\in \R$.
One can readily verify that the network is not exponentially stable in the absence of inputs. Moreover, for arbitrarily small inputs the network may exhibit unbounded state trajectories.
Finally, condition~\eqref{eq_uniformity_gammaiu} is also crucial for ISS of the overall system. Consider a network composed of subsystems of the form 
\begin{eqnarray*}
  {x}_i^+ = i u , \quad x_i,u\in \R, \, i \in\N .
\end{eqnarray*}
It is easy to verify that this network is not ISS.

Condition~\eqref{eq_uniformity_gamma} has been used in \cite{DaP20} for the analysis of nonlinear infinite networks via the small-gain approach.
It introduces some conservatism into small-gain analysis, compared with conditions for the gains, used in small-gain analysis of finite networks. Nevertheless, in contrast to \cite{DaP20} we use finite-step Lyapunov-like functions, which allows to significantly reduce the conservativeness of our analysis.


As we frequently work with distances, we state the following useful
sufficient condition for well-posedness of $\Sigma$, based on
Lemma~\ref{lem:well-posedness}. 
\begin{lemma}
\label{lem:Uniform-K-boundedness} 
Assume that all $f_i$ are uniformly $\K$-bounded with respect to $\A_i$,
that is,  there are $\kappa_1,\kappa_2\in\Kinf$ such that
one of the following conditions 
\begin{equation}
|f_i (\xi_i,\ol \xi_i,\mu_i)|_{\A_i}
\leq
\kappa_1 (|\xi_i|_{\A_i}) + \kappa_2 (|\ol \xi_i|) + \kappa_2 (|\mu_i|),
\label{eq:Well-posedness-lemma-A-i-1}
\end{equation}
\begin{equation}
|f_i (\xi_i,\ol \xi_i,\mu_i)|_{\A_i}
\leq
\kappa_1 (|\xi_i|_{\A_i}) + \kappa_2 (|\ol \xi_i|_{ \A(I_i)}) + \kappa_2 (|\mu_i|),
\label{eq:Well-posedness-lemma-A-i-2}
\end{equation}
holds for all $i\in\N$, all $\xi_i\in \R^{n_i}$, $\ol\xi_i \in  X({I_i})$ and $\mu_i\in \R^{p_i}$.
Further assume that $(\A_i)_{i\in\N}$ is uniformly bounded. Then $\Sigma$ is well-posed.
\end{lemma}

\begin{proof}
First let \eqref{eq:Well-posedness-lemma-A-i-1} hold and pick any $i\in\N$.
By the $\K$-boundedness of the $f_i$ and the definition of the distance
function $|\cdot|_{\A_i}$, we have for all $\xi_i\in \R^{n_i}$, $\ol\xi_i \in  X({I_i})$ and $\mu_i\in \R^{p_i}$ 
\begin{align*}
|f_i (\xi_i,\ol \xi_i,\mu_i)| &- \sup_{y_i \in \A_i} |y_i|  \leq |f_i (\xi_i,\ol \xi_i,\mu_i)|_{\A_i} \\
&\leq \kappa_1 (|\xi_i|_{\A_i}) + \kappa_2 (|\ol \xi_i|) + \kappa_2 (|\mu_i|) \\
&\leq \kappa_1 (|\xi_i|+\sup_{y_i \in \A_i} |y_i|) + \kappa_2 (|\ol \xi_i|) + \kappa_2 (|\mu_i|) .
\end{align*}
As $\kappa_1(a+b)\leq \kappa_1(2a)+\kappa_1(2b)$ for all $a,b\geq 0$, and in view of the uniform boundedness of $(\A_i)_{i\in\N}$, we obtain for all $\xi_i\in \R^{n_i}$, $\ol\xi_i \in  X({I_i})$ and $\mu_i\in \R^{p_i}$ that
\begin{align*}
|f_i (\xi_i,\ol \xi_i,&\mu_i)| \leq \kappa_1 (|\xi_i|+C) + \kappa_2 (|\ol \xi_i|) + \kappa_2 (|\mu_i|)  + C\\
& \leq \kappa_1 (2|\xi_i|) +\kappa_1 (2C) + \kappa_2 (|\ol \xi_i|) + \kappa_2 (|\mu_i|)  + C .
\end{align*}
Taking $\kappa := \max  \{\kappa_1 \circ 2\id,\kappa_2\}$ and $\ol C := \kappa_1 (2C) + C$ gives
\begin{align}
\label{eq:Uniform-K-boundedness-final-1} 
|f_i (\xi_i,\ol \xi_i,\mu_i)| \leq \ol C + \kappa (|\xi_i|)  + \kappa (|\ol \xi_i|)  + \kappa (|\mu_i|) ,
\end{align}
and $\Sigma$ is well-posed by Lemma~\ref{lem:well-posedness}.

Now assume~\eqref{eq:Well-posedness-lemma-A-i-2}. Repeating the steps from above, we have the following counterpart of  \eqref{eq:Uniform-K-boundedness-final-1} (with the same $\ol C,\kappa$):
\begin{align}
\label{eq:Uniform-K-boundedness-final-2} 
|f_i (\xi_i,\ol \xi_i,\mu_i)| \leq \ol C + \kappa (|\xi_i|)  + \kappa (|\ol \xi_i|_{ \A(I_i)})  + \kappa (|\mu_i|).
\end{align}
In view of \eqref{eq:Distance-in-overline-A_i} and the uniform boundedness of $(\A_i)_{i\in\N}$, it follows that 
\begin{eqnarray*}
\kappa (|\ol \xi_i|_{ \A(I_i)}) &=& \kappa (\sup_{j\in I_i}|\xi_j|_{ \A(I_j)})
\leq \kappa (\sup_{j\in I_i}(|\xi_j| + C))\\
&\leq& \kappa (\sup_{j\in I_i}|\xi_j| + C) = \kappa (|\ol \xi_i| + C)\\
&\leq&  \kappa (2|\ol \xi_i|) + \kappa(2 C).
\end{eqnarray*}
Using this estimate in \eqref{eq:Uniform-K-boundedness-final-2}, and applying Lemma~\ref{lem:well-posedness},
we see that $\Sigma$ is well-posed.
\end{proof}

We have the following well-posedness result.
\begin{theorem}
\label{thm:Well-posedness-coupled-systems}
If Assumption~\ref{ass:Properties-of-Sigma_i} holds, $(\A_i)_{i\in\N}$ is uniformly bounded and Assumptions~\ref{ass_vi_existence} and~\ref{ass_external_gains} are satisfied, then $\Sigma^M$ is well-posed.
\end{theorem}
\begin{proof}
From Assumptions~\ref{ass_vi_existence} and~\ref{ass_external_gains} we have that 
\begin{align*}
 \ul{\omega} (|x_i(M,&\xi,u)|_{\A_i}) \leq 
 \ul{\omega}_i (|x_i(M,\xi,u)|_{\A_i}) \leq W_i (x_i(M,\xi,u))\\
&\leq \!\max \!\big\{\! \sup_{j \in \N} \gamma_{ij} (W_j(\xi_j)) , \gamma_{iu} (\|u\|_{\infty}) \big\}\\
&\leq \!\max \!\big\{\! \sup_{j \in \N} \gamma_{ij} (\ol{\omega}_j (|\xi_j|_{\A_j})) , \gamma_{iu} (\|u\|_{\infty}) \big\}\\
&\red{\leq} \!\max \!\big\{\! \sup_{j \in I_i(M)\cup\{i\}} \gamma_{ij} (\ol{\omega} (|\xi_j|_{\A_j})) , \gamma_{iu} (\|u\|_{\infty}) \big\}\\
&\leq \!\max \!\big\{\! \sup_{j \in I_i(M)\cup\{i\}} \alpha (\ol{\omega} (|\xi_j|_{\A_j})) ,  \ol{\gamma}_u (\|u\|_{\infty}) \big\}.
\end{align*}
Defining $\kappa:=\max\{\ul{\omega}^{-1}\circ\alpha\circ \ol{\omega}, \ul{\omega}^{-1}\circ \ol{\gamma}_u\}$, it holds that: 
\begin{align*}
|x_i(M,\xi,u)|_{\A_i}
&\leq \!\max \!\big\{\! \sup_{j \in I_i(M)\cup\{i\}} \kappa (|\xi_j|_{\A_j}) , \kappa  (\|u\|_{\infty}) \big\} \\
&\leq \kappa (|\xi_i|_{\A_i})  +  \kappa (\sup_{j \in I_i(M)}|\xi_j|_{\A_j}) + \kappa  (\|u\|_{\infty}) \\
\end{align*}
In view of \eqref{eq:Distance-in-overline-A_i} this is precisely 
the estimate \eqref{eq:Well-posedness-lemma-A-i-2} but for the map $f^M$ instead of $f$.
This shows the well-posedness of $\Sigma^M$ by Lemma~\ref{lem:Uniform-K-boundedness}.
\end{proof}


\begin{example}
\label{examp:Example-well-posedness} 
Here we show that well-posedness of the network does not necessarily follow from the validity of Assumptions~\ref{ass:Properties-of-Sigma_i},~\ref{ass_vi_existence} and~\ref{ass_external_gains} with $M=1$, if $(\A_i)_{i\in\N}$ is not uniformly bounded.

Let $\A_i:=[-i,i] \subset \R$, $i\in\N$ and let $\Sigma_i$ be given by equations:
\[
x_i^+ = f_i(x_i),
\]
where
\[
f_i(x_i) = 
\begin{cases}
x_i - \frac{1}{2}|x_i|_{\A_i}\sgn{(x_i)}, & \text{ if } x_i\not\in \A_i, \\ 
i x_i, & \text{ if } x_i \in [-\frac{1}{2},\frac{1}{2}], \\ 
g_i(x_i), & \text{ otherwise, }
\end{cases}
\]
where $g_i$ is chosen in a way that $\A_i$ is invariant w.r.t. $f_i$, and $f_i$ is continuous. This shows that 
Assumption~\ref{ass:Properties-of-Sigma_i} is satisfied.

Pick $W_i(x_i)=|x_i|_{\A_i}$. Then $W_i(x_i^+)=\frac{1}{2}W_i(x_i)$, and thus  Assumption~\ref{ass_vi_existence} holds with
 $\gamma_{ii}=\frac{1}{2}$, and $\gamma_{ij}=0$ otherwise.
Thus, each $\Sigma_i$ is exponentially ISS with respect to $\A_i$, $W_i$ is an eISS Lyapunov function for $\Sigma_i$ with respect to $\A_i$ and  
Assumptions~\ref{ass_vi_existence} and~\ref{ass_external_gains} are satisfied with $M=1$.

However, choosing $x^0:=\frac{1}{4}(1,1,\ldots) \in \A$, we see that $f (x^0) = \frac{1}{4}(1,2,\ldots,i,\ldots) \in \big\{\prod_{i\in\N}\A_i\big\} \backslash X$, i.e. $f_i$ is not well-defined at this point (and in all points in the interior of $[-\frac{1}{2},\frac{1}{2}]^\N$ in $X$).
\end{example}

Now we establish that the interconnected system $\Sigma$ is ISS under the given assumptions.
By Proposition~\ref{prop_iss}, our objective is to find a finite-step ISS Lyapunov function for the infinite network $\Sigma$.
This is achieved by the following \emph{relaxed} small-gain theorem.

\begin{theorem}\label{MT}
Consider a well-posed infinite network $\Sigma =(f,X,U)$ and nonempty closed sets $\A_i \subset \R^{n_i}$, $i\in\N$.
Let the set $\A = X \cap \prod_{i\in\N}\A_i$ be nonempty.
Suppose that Assumptions~\ref{ass_vi_existence} and~\ref{ass_external_gains} are satisfied.
Then $\Sigma$ admits an $M$-step ISS Lyapunov function with respect to~$\A$ of the form%
\begin{equation}\label{eq:Lyapunov-function-construction}
  V(\xi) = \sup_{i\in\N}  W_i(\xi_i) , \quad V:X \rightarrow \R_+ .%
\end{equation}
In particular, the function $V$ has the following properties.%
\begin{enumerate}[(i)]
\item For all $\xi \in X$ and $u \in \UC$%
\begin{equation}\label{eq:overall-V-decay}
V(x(M,\xi,u)) \leq \max \big\{ \alpha (V(\xi) ) , \ol{\gamma}_u (\|u\|_{\infty}) \big\} .%
\end{equation}
\item For every $\xi \in X$ the following inequalities hold:%
\begin{equation}\label{eq:Coercivity-bound-for-V}
\ul{\omega}(|\xi|_\A) \leq  V(\xi) \leq \ol{\omega}(|\xi|_\A) .%
\end{equation}
\end{enumerate}
In particular, $\Sigma$ is ISS with respect to $\A$.%
\end{theorem}
\begin{proof}
From the second inequalities~\eqref{eq_viest} and~\eqref{eq_uniformity_alpha}, we have
\begin{align*}
  V(\xi) &= \sup_{i\in\N}  W_i(\xi_i)  \leq \sup_{i\in\N} \ol\omega_i (\abs{\xi_i}_{\A_i})\leq \sup_{i\in\N} \ol\omega (\abs{\xi_i}_{\A_i}) .
\end{align*}
It follows from the monotonicity of $\ol\omega$ and Lemma~\ref{lem:Alternative-A-representation} that
\begin{equation*}
V(\xi) \leq \ol\omega (\sup_{i\in\N} \abs{\xi_i}_{\A_i}) = \ol\omega (\abs{\xi}_\A) .
\end{equation*}
This shows that $V$ is well-defined (i.e. finite valued) and gives an upper bound in 
\eqref{eq:Coercivity-bound-for-V}.
Similarly, it holds that
\begin{equation*}
V(\xi) \geq \ul\omega (\sup_{i\in\N} \abs{x_i}_{\A_i}) = \ul\omega (\abs{\xi}_\A).
\end{equation*}

Given any $\xi \in X$ and $u\in\UC$, we also have that
\begin{align*}
& \hspace{-1.2cm} V (x(M,\xi,u)) 
=  \sup_{i\in\N}  W_i(x_i(M,\xi,u)) \\
\qquad \stackrel{ \eqref{eq_nablaviest}}{\leq} & \sup_{i\in\N} \max\big\{\sup_{j\in\N}\gamma_{ij}\big(W_j(\xi_j)\big),\gamma_{iu}(\norm{u}_{\infty})\big\} \\
\qquad \stackrel{ \eqref{eq_uniformity_gamma}}{\leq} &\sup_{i\in\N} \max \big\{\sup_{j\in\N} \alpha  \big(W_j (\xi_j)\big) , \gamma_{iu} (\norm{u}_{\infty}) \big\} \\
\qquad = & \sup_{i\in\N}  \max \big\{ \alpha  \big(\sup_{j\in\N} W_j (\xi_j)\big) , \gamma_{iu} (\norm{u}_{\infty}) \big\} \\
\qquad = & \max \big\{ \alpha  \big(V (\xi)\big) , \sup_{i\in\N}\gamma_{iu} (\norm{u}_{\infty}) \big\} \\
\qquad \stackrel{ \eqref{eq_uniformity_gammaiu}}{\leq} & \max \big\{ \alpha  \big(V (\xi)\big) , \ol\gamma_{u} (\norm{u}_{\infty}) \big\} ,
\end{align*}
which is identical to~\eqref{eq:overall-V-decay}. Hence $V$ is a finite-step Lyapunov function for $\Sigma$ with respect to~$\A$ and by Proposition~\ref{prop_iss} $\Sigma$ is ISS in~$\A$.
\end{proof}

\begin{remark}
\label{rem:Continuity-of-LF-for-composite-system} 
Continuity of $V$ can be inferred from continuity of $W_i$, provided all $W_i$ have 
a uniform modulus of a continuity on bounded balls. For example, this holds if all $W_i$ are 
locally Lipschitz continuous on bounded balls with a uniform in $i$ Lipschitz constant.
\end{remark}


In the special case $M=1$, Theorem~\ref{MT} reduces to a classic small-gain theorem, and it is a discrete-time counterpart of~\cite[Theorem 1]{DaP20}.
Note that in the case $M=1$ the well-posedness of $\Sigma$
 is no more an assumption in Theorem~\ref{MT}, as it follows from other assumptions of Theorem~\ref{MT} by invocation of Theorem~\ref{thm:Well-posedness-coupled-systems}.
We also note that the conservatism of Theorem~\ref{MT} basically comes from condition~\eqref{eq_uniformity_alpha} which demands that all the coupling gains $\gamma_{ij}$'s have to be less than identity.
Nevertheless, for exponentially ISS systems, we are able to establish the \emph{necessity} of our small-gain theorem, which shows the \emph{non}-conservatism of the proposed small-gain condition in this case.

\begin{theorem}\label{nec}
Consider a well-posed infinite network $\Sigma$.
Let $\Sigma$ be eISS with respect to $\A = X \cap \prod_{i\in\N}\A_i$, with  nonempty closed sets $\A_i \subset \R^{n_i}$.
Then there exist continuous functions $W_i:\R^{n_i} \rightarrow \R_+$, $i \in \N$ and $M \in \N$ such that Assumptions~\ref{ass_vi_existence} and~\ref{ass_external_gains} hold.
\end{theorem}
	\begin{proof}
		The eISS property of the system $\Sigma$ implies that there exist $M \in \N$, $\gamma \in \KC$ and $c<1$ such that for any initial state $\xi \in X$ and any $u\in\UC$ we have
		\begin{align}\label{eiss}
		|x(M,\xi,u)|_\A \leq \max\big\{c|\xi|_\A , \gamma(\|u\|_{\infty}) \big\}.
		\end{align}
		Define the function $W_i(x_i):=|x_i|_{\A_i}$, $V_i:\R^{n_i} \rightarrow \R_+$ for $i \in \N$.
		This choice of $W_i$ clearly satisfies~\eqref{eq_viest} and~\eqref{eq_uniformity_alpha}	with $b=1$, $\ul\omega = \ul\omega_i = \ol\omega_i = \ol\omega = 1$.
We also have
\begin{align}\label{final4}
W_i(x_i({M},\xi,u ))=& |x_i({M},\xi,u )|_{\A_i} \leq |x({M},\xi,u )|_{\A} \nonumber\\
 \mathop  \le \limits^{(\ref{eiss})}& \max\big\{ c |\xi|_{\A} , \gamma(\|u\|_{\infty})\big\} \nonumber\\
 =& \max\big\{c\sup_{j\in\N}|{\xi_j}|_{\A_j} , \gamma(\|u\|_{\infty})\big\} \nonumber\\
 =& \max\big\{\sup_{j\in\N}c|{\xi_j}|_{\A_j} , \gamma(\|u\|_{\infty})\big\} .
\end{align}
We can rewrite~\eqref{final4} as
\begin{align}\label{final5}
W_i (x_i({M},\xi,u )) \leq \max\big\{\sup_{j\in\N}c W_j(\xi_j) , \gamma(\|u\|_{\infty})\big\} ,
\end{align}
which implies that~\eqref{eq_nablaviest} is satisfied with $\gamma_{ij} = c$ and $\gamma_{iu} = \gamma$ for all $i,j \in \N$.
			Conditions~\eqref{eq_uniformity_gamma} and~\eqref{eq_uniformity_gammaiu} also hold with $\alpha = c$ and $\ol\gamma_{u} = \gamma$, which completes the proof.
\end{proof}

Note that in the proof of Theorem~\ref{nec} we \emph{explicitly} construct the individual ISS Lyapunov-like functions $W_i$ and the corresponding gain functions.
Moreover, there is a positive integer $M$ for which all the coupling gain functions $\gamma_{ij}$ can be chosen \emph{identically}.
In the spirit of our recent work~\cite{Noroozi.2020a}, we believe that one can gain from these observations for a distributed control design based on our small-gain theorem.
This is our ongoing research work.

\section{From Infinite to Finite Networks}\label{sec:From-infinite-to-finite-networks}

The underlying idea to deal with infinite networks is to develop tools for analysis and design of arbitrarily large-but-finite networks, which are independent of the possibly \emph{unknown} size of the network.
The question arises whether the quantitative stability indices, e.g. decay rate of solutions, obtained for an over-approximating infinite network are preserved for the original large-but-finite network.
This section addresses this question. We show that the quantitative ISS
indices given by Theorem~\ref{MT} will be preserved for any truncation of
an infinite network if $M=1$ in Assumption~\ref{ass_vi_existence}.
For $M>1$, which is a more complex case, we provide a scale-free ISS stabilization
result.

For the purpose of the truncation process, we only consider the
  first $n \in \N$ subsystems of $\Sigma$ and denote the truncated system
  by $\Sigma^{\langle n \rangle}$.  As the states $x_j$ for $j>n$ are no
  longer present in $\Sigma^{\langle n \rangle}$, but \emph{in general}
  may still appear in some of the equations, we interpret these $x_j$ as
  additional \emph{external} inputs. We denote by
  \begin{equation*}
      I^{\langle n \rangle} := \bigcup_{i=1}^n I_i \setminus \{ 1,\ldots,n
      \}  
  \end{equation*}
  the set of neighbors of the first $n$ systems. By
  Assumption~\ref{ass:Properties-of-Sigma_i}\,\eqref{itm:gen-assumption-2}
  the set $I^{\langle n \rangle}$ is finite.
  The truncation of the infinite network
  is represented by
\begin{align}
\label{eq_interconnection-truncated}
  \Sigma^{\langle n \rangle}:\quad (x^{\langle n \rangle})^+ = f^{\langle n \rangle} ( x^{\langle n \rangle} , \tilde x, u^{\langle n \rangle} ),
\end{align}
%
%
%
%
where the state vector $x^{\langle n \rangle} = (x_i)_{1\leq
    i\leq n} \in \R^N$, $N := \sum_{i=1}^n n_i$, the input vector
  $u^{\langle n \rangle} = (u_i)_{1\leq i\leq n}\in \R^P$, $P :=
  \sum_{i=1}^n p_i$, the additional input vector 
$\tilde x := (x_j)_{j \in I^{\langle n \rangle}} \in \R^L$, $L:=
\sum_{j\in I^{\langle n \rangle}} n_j$ and the dynamics 
$f^{\langle n \rangle}= (f_i)_{1\leq i\leq n}$, $f^{\langle n
  \rangle}:\R^N\tm \R^L \tm \R^P \to\R^N$. 

The network $ \Sigma^{\langle n \rangle}$ is obtained by keeping the first $n$ subsystems of the infinite network $\Sigma$ together with the associated interconnection between these subsystems. 
Note that in our formulation we do not neglect  the other subsystems $\Sigma_i$ with $i > n$ (i.e. $\tilde x$), but instead we view them as additional inputs to the network $ \Sigma^{\langle n \rangle}$.
Clearly removal of all subsystems $\Sigma_i, i > n$ is covered by our formulation as a special case if one sets $\tilde x \equiv 0$.

In the following we derive conditions under which $\Sigma^{\langle n \rangle}$ is
  ISS in the set $\A^{\langle n \rangle} := \Pi_{i=1}^n \A_i$ and compute the corresponding ISS gain functions under assumption that $\Sigma$ is ISS.

\begin{theorem}\label{thm:truncation}
Consider the infinite network $\Sigma$.
Suppose that Assumption~\ref{ass_vi_existence} with $M=1$ and Assumption~\ref{ass_external_gains} are satisfied.
Then for each $n\in\N$ the function $V^{\langle n \rangle} : \R^N \to \R_+$ defined by
\begin{equation}\label{eq:Lyapunov-function-construction-truncated}
  V^{\langle n \rangle}(\xi^{\langle n \rangle}) = \max_{1\leq i\leq n}  W_i(\xi_i) ,
\end{equation}
satisfies for all $\xi^{\langle n \rangle} \in \R^N,\tilde \xi \in X$ and $u^{\langle n \rangle} \in \R^P$:
\begin{align}
&\ul{\omega}(|\xi|_{\A^{\langle n \rangle}}) \leq  V^{\langle n \rangle} (\xi^{\langle n \rangle}) \leq \ol{\omega}(|\xi|_{\A^{\langle n \rangle}}) , \label{eq:Coercivity-bound-for-V-truncated}\\
&V^{\langle n \rangle}(f^{\langle n \rangle}(\xi^{\langle n \rangle},\tilde \xi,u^{\langle n \rangle})) \leq \max \big\{ \alpha (V^{\langle n \rangle}(\xi^{\langle n \rangle}) ) ,  \alpha\circ\ol\omega ( |{\tilde \xi}|_\infty) , \nonumber\\
&\qquad\qquad\qquad\qquad\qquad\qquad\ol{\gamma}_u (\|u^{\langle n \rangle}\|_{\infty}) \big\} .
\label{eq:overall-V-decay-truncated}
\end{align}
Thus, $V^{\langle n \rangle}$ is a 1-step ISS Lyapunov function for $\Sigma^{\langle n \rangle}$ and hence 
$\Sigma^{\langle n \rangle}$ is ISS with respect to~$\A^{\langle n \rangle}$.
\end{theorem}

\begin{proof}
The proof of~\eqref{eq:overall-V-decay-truncated} follows the same arguments as those for~\eqref{eq:overall-V-decay} in Theorem~\ref{MT}.

The Assumption~\ref{ass_vi_existence} for $M=1$ ensures that for all $i\in\{1,\ldots,n\}$ and all $j\in\N$ there are $\gamma_{ij} \in  \Kinf \cup \{0\}$ and $\gamma_{iu} \in \K$ so that for all $\xi = (\xi^{\langle n \rangle},(\tilde{\xi}_j)_{j\geq n}) \in X$, all $u^{\langle n \rangle} \in \UC^{\langle n \rangle}:=\ell_\infty(\N,\R^P)$  the following holds 
%
\begin{align*}
& W_i(x_i(1,\xi,u)) = W_i (f_i(\xi_i,\bar{\xi}_i,u_i)) \nonumber\\
&\!\leq \!\max \!\big\{\! \sup_{1\leq j\leq n} \gamma_{ij} (W_j(\xi_j)),
 \sup_{j \geq n} \gamma_{ij} (W_j(\xi_j)) , \gamma_{iu} (\|u^{\langle n \rangle}\|_{\infty}) \big\}, \nonumber\\
& \stackrel{ \eqref{eq_uniformity_gamma}, \eqref{eq_uniformity_gammaiu}}{\leq} \! \!\!\!\max \!\big\{\! \max_{1\leq j  \leq n} \!\alpha (W_j(\xi_j)) ,
  \alpha (\sup_{j\ge n} W_j(\tilde \xi_j)),\!\ol\gamma_{u} (\|u^{\langle n \rangle}\|_{\infty}) \big\}\nonumber\\
& \stackrel{ \eqref{eq_uniformity_alpha}}{\leq} \! \!\max \!\big\{\! \max_{1\leq j \leq n} \!\alpha (W_j(\xi_j)) , 
  \alpha\circ\ol\omega (|\tilde \xi|_\infty),\ol\gamma_{u} (\|u^{\langle n \rangle}\|_{\infty}) \big\} .
\end{align*}
This shows that $W_i$ are (1-step) ISS Lyapunov functions for subsystems of $\Sigma^{\langle n\rangle}$. Moreover, it follows from~\eqref{eq:Lyapunov-function-construction-truncated} that \eqref{eq:overall-V-decay-truncated} holds. 
%
Hence $ V^{\langle n\rangle}$ is a 1-step ISS Lyapunov function for $\Sigma^{\langle n\rangle}$ and by \cite[Theorem 7]{Noroozi.2018a} we conclude ISS of $\Sigma^{\langle n\rangle}$ in the set $\A^{\langle n\rangle}$. 
\end{proof}

As seen  from~\eqref{eq:overall-V-decay-truncated}, the \emph{decay} rate $\alpha$ is \emph{preserved} under the truncation. Moreover, if the additional external inputs $\tilde x$ are not present to $\Sigma^{\langle n\rangle}$, the input gain $\ol\gamma_u$ is preserved.
Therefore, stability/performance indices of the overall system will be \emph{independent} of the size of the network and we obtain the scale-free ISS for all truncations of $\Sigma$. We further illustrate this aspect via numerical simulation below.

As we assumed that $M=1$, the dynamics of the modes of the truncated systems are the same as the dynamics of the corresponding modes of the infinite network, up to the fact that some states of the infinite-dimensional system become the inputs for the truncated system.
This easily implies that the functions $W_i$, which are ISS Lyapunov functions for subsystems of $\Sigma$ are also ISS Lyapunov functions for subsystems of the truncated system $\Sigma^{\langle n\rangle}$.

For $M>1$ the situation is more complex.
Firstly, if $M>1$, then some subsystems may be unstable, and thus clearly there exist truncations of the infinite network, which are not ISS. Secondly, the dynamics of the modes of $(f^{\langle n\rangle})^{M}$ are different from the dynamics of $(f^{M})^{\langle n\rangle}$.
However, for $M>1$ we have the following scale-free ISS stabilization
result. 
\begin{theorem}
\label{thm:Scale-free-stabilizability-of-truncations} 
Let $M>1$ and suppose that, with this $M$,
  Assumption~\ref{ass_vi_existence} and
  Assumption~\ref{ass_external_gains} are satisfied.  Let $n\in\N$ be
  arbitrary. Then there are $\beta\in\KL$ and $\gamma\in\Kinf$, such that
  for all $\xi^{\langle n \rangle}_0 \in \R^N$ and $u^{\langle n
    \rangle}(\cdot) \in \UC^{\langle n \rangle}:=\ell_\infty(\N,\R^P)$
  there is a globally bounded signal $\tilde{x}(\cdot) := \tilde{x}(\cdot,
  \xi^{\langle n \rangle}, u^{\langle n \rangle}): \N_0:\to \R^M$, such
  that it holds for all $k \in \N_0$ that
\begin{multline}
\label{eq:ISS-inf-network}
  \hspace{-1mm}|x^{\langle n \rangle}(k;\xi^{\langle n \rangle}_0,\tilde{x}(\cdot),u^{\langle n
  \rangle}(\cdot))|_{\A^{\langle n \rangle}} \\ \leq \max\big\{ \beta( |\xi^{\langle n \rangle}_0|_{\A^{\langle n \rangle}},k), \gamma(\|u^{\langle n \rangle}\|_{\infty}) \big\}.
\end{multline}
\end{theorem}
In the previous theorem, $\beta$ and $\gamma$ do not depend on $n$, which shows a  \q{scale-free} ISS stabilizability property 
for all truncations of $\Sigma$.
\begin{proof}
Suppose that Assumptions~\ref{ass_vi_existence} and~\ref{ass_external_gains} are satisfied.
By Theorem~\ref{MT}, the system $\Sigma$ is ISS with respect to $\A$.
In particular, there exist $\beta$ and $\gamma$ independent of $n$ such that for any $\xi = (\xi^{\langle n \rangle}_0,(\tilde{\xi}_j)_{j\geq n+1}) \in X$ with 
$(\tilde{\xi}_j)_{j\geq n+1} \in \prod_{i\geq n+1}\A_i$, and all $u:=(u^{\langle n \rangle},0)\in\UC$ the corresponding solution of the system $\Sigma$ satisfies
\begin{equation}
\label{eq:ISS-inf-network-Sigma-estimate}
  \hspace{-1mm}|x(k,\xi,u)|_\A \leq \max\big\{ \beta( |\xi^{\langle n \rangle}_0|_{\A^{\langle n \rangle}},k), \gamma(\|u^{\langle n \rangle}\|_{\infty}) \big\} \,, k \in \N_0 .%
\end{equation}
This holds as $|\xi^{\langle n \rangle}_0|_{\A^{\langle n \rangle}} = |\xi|_{\A}$ and $\|u^{\langle n \rangle}\|_{\infty}=\|u\|_{\infty}$.


For $\xi^{\langle n \rangle} \in \R^N$ and $u^{\langle n \rangle} \in \UC^{\langle n \rangle}:=\ell_\infty(\N,\R^P)$ 
define $\tilde{x}(k, \xi^{\langle n \rangle}, u^{\langle n
  \rangle}):=(x_i(k,\xi,u))_{i\in I^{\langle n \rangle}}$, where the $x_i$
are the components of the solution of $\Sigma$ corresponding to the
initial condition $\xi$ and the input $u$.
By \eqref{eq:ISS-inf-network-Sigma-estimate}, we have for all $k \in \N_0$ that 
\begin{equation}
\label{eq:Boundedness of controls}
  \hspace{-1mm}|\tilde{x}(k)|_{\prod_{i\in I^{\langle n \rangle}}\A_i} \leq \max\big\{ \beta( |\xi^{\langle n \rangle}|_{\A^{\langle n \rangle}},k), \gamma(\|u^{\langle n \rangle}\|_{\infty}) \big\}.
\end{equation}
This shows that $\tilde{x}(\cdot)$ is bounded. In addition,
\eqref{eq:ISS-inf-network} follows from \eqref{eq:ISS-inf-network-Sigma-estimate}, as $x^{\langle n \rangle}(\cdot;\xi^{\langle n \rangle}_0,\tilde{x}(\cdot),u^{\langle n  \rangle}(\cdot))$ just consists of the first $n$ components of the 
trajectory $x(\cdot;\xi,u)$ of $\Sigma$.
\end{proof}

\begin{remark}
    Theorem~\ref{thm:truncation} only considers the case $M=1$ in
    Assumption~\ref{ass_vi_existence}, i.e. classic small-gain conditions.
    For larger $M$ conditions of the interplay of the finite truncation
    with the remainder of the network are required, as this remainder may
    have a stabilizing effect on the finite network. 
\end{remark}

\section{Illustrative Example}


In this section, we verify the effectiveness of our small-gain theorem by application to the control of traffic networks.

We revisit an example of a traffic network composed of infinitely many
cells, indexed by $i\in\N$, which was considered in~\cite{Kibangou}, \cite{KMS19}.
Each cell $i$ represents a continuous-time system $\Sigma_i$ described by
\begin{align}\label{subsys}
  \Sigma_i : \dot x_i = -\Bigl(\frac{v_i}{l_i} + e_i\Bigr)x_i + D_i\ol{x}_i + B_iu_i,%
\end{align}
with $x_i,u_i\in \R$ and the following structure%
\begin{enumerate}
\item[$-$] $e_i=0,D_i=c\frac{v_{i+1}}{l_{i+1}},\ol{x}_i=x_{i+1},B_i=0$ if $i\in S_1:=\{1\}$;
\item [$-$]$e_i=0,D_i=c\frac{v_{{i+4}}}{l_{i+4}},\ol{x}_i=x_{i+4},B_i=r>0$ if $i\in S_2:=\{4+8j : j\in\N\cup\{0\}\}$;
\item [$-$]$e_i=0,D_i=c\frac{v_{{i-4}}}{l_{i-4}},\ol{x}_i=x_{i-4},B_i=\frac{r}{2}$ if $i\in S_3:=\{5+8j :j\in\N\cup\{0\}\}$;
\item [$-$]$e_i=0,D_i=c(\frac{v_{i-1}}{l_{i-1}},\frac{v_{i+4}}{l_{i+4}})\trn,\ol{x}_i=(x_{i-1},x_{i+4}),B_i=0$ if $i\in S_4:=\{6+8j : j\in\N\cup\{0\}\}$;
\item [$-$] $e_i=e\in(0,1),D_i=c(\frac{v_{i-4}}{l_{i-4}},\frac{v_{i+1}}{l_{i+1}})\trn,\ol{x}_i=(x_{i-4},x_{i+1}),B_i=0$ if $i\in S_5:=\{9+8j : j\in\N\cup\{0\}\}$;
\item [$-$] $e_i=0,D_i=c(\frac{v_{i+1}}{l_{i+1}},\frac{v_{i+4}}{l_{i+4}})\trn,\ol{x}_i=(x_{i+1},x_{i+4}),B_i=0$ if $i\in S_6:=\{2+8j : j\in\N\cup\{0\}\}$;
\item [$-$] $e_i=0,D_i=c(\frac{v_{i-4}}{l_{i-4}},\frac{v_{i-1}}{l_{i-1}})\trn,\ol{x}_i=(x_{i-4},x_{i-1}),B_i=0$ if $i\in S_7:=\{7+8j : j\in\N\cup\{0\}\}$;
\item [$-$] $e_i=2e,D_i=c(\frac{v_{i-1}}{l_{i-1}},\frac{v_{i+4}}{l_{i+4}})\trn,\ol{x}_i=(x_{i-1},x_{i+4}),B_i=0$ if $i\in S_8:=\{8+8j :j\in\N\cup\{0\}\}$;
\item [$-$] $e_i=0,D_i=c(\frac{v_{i-4}}{l_{i-4}},\frac{v_{i+1}}{l_{i+1}})\trn,\ol{x}_i=(x_{i-4},x_{i+1}),B_i=0$ if $i\in S_9:=\{11+8j : j\in\N\cup\{0\}\}$;
\item[$-$] $e_i=0,D_i=c\frac{v_{i+1}}{l_{i+1}},\ol{x}_i=x_{i+1},B_i=r/2$ if $i\in S_{10} :=\{3\}$;
\end{enumerate}
where, for all $i\in\N$, $0< \underaccent{\bar}{v} \leq v_i\leq \ol{v}$, $0<\ul{l}\leq l_i\leq \ol{l}$, and $c\in(0,0.5)$.
In~\eqref{subsys}, $l_i$ is the length of a cell in kilometers (km), and $v_i$ is the flow speed of the vehicles in kilometers per hour (km/h).
The state $x_i$ is the density of traffic, given in vehicles per cell, for each cell $i$ of the road.
The scalars $B_i$ represent the number of vehicles that can enter the cells through entries which are controlled by $u_i$, with $u_i=1$ and $u_i=0$ correspond to green and red light, respectively.
The percentage of vehicles leaving the cells using available exits is denoted by $e_i$.
Furthermore, $c$, which is a design parameter, reflects the percentage of vehicles entering cell $i$ from the neighboring cells.
Such a traffic network schematically is illustrated by Figure~\ref{allr}.
Discretizing system~\eqref{subsys} over time, each cell $i$ in discrete-time is described by
\begin{align}\label{subsys-d}
  \Sigma_i^d: x_i^+ = \Big(1-T\big(\frac{v_i}{l_i} + e_i\big)\Big) x_i + T D_i\ol{x}_i + T B_i u_i,
\end{align}
where $T > 0$ is the sampling period.

\begin{figure}
	\vspace*{-0.0cm}
	\begin{center}
			\includegraphics[height=5.5cm]{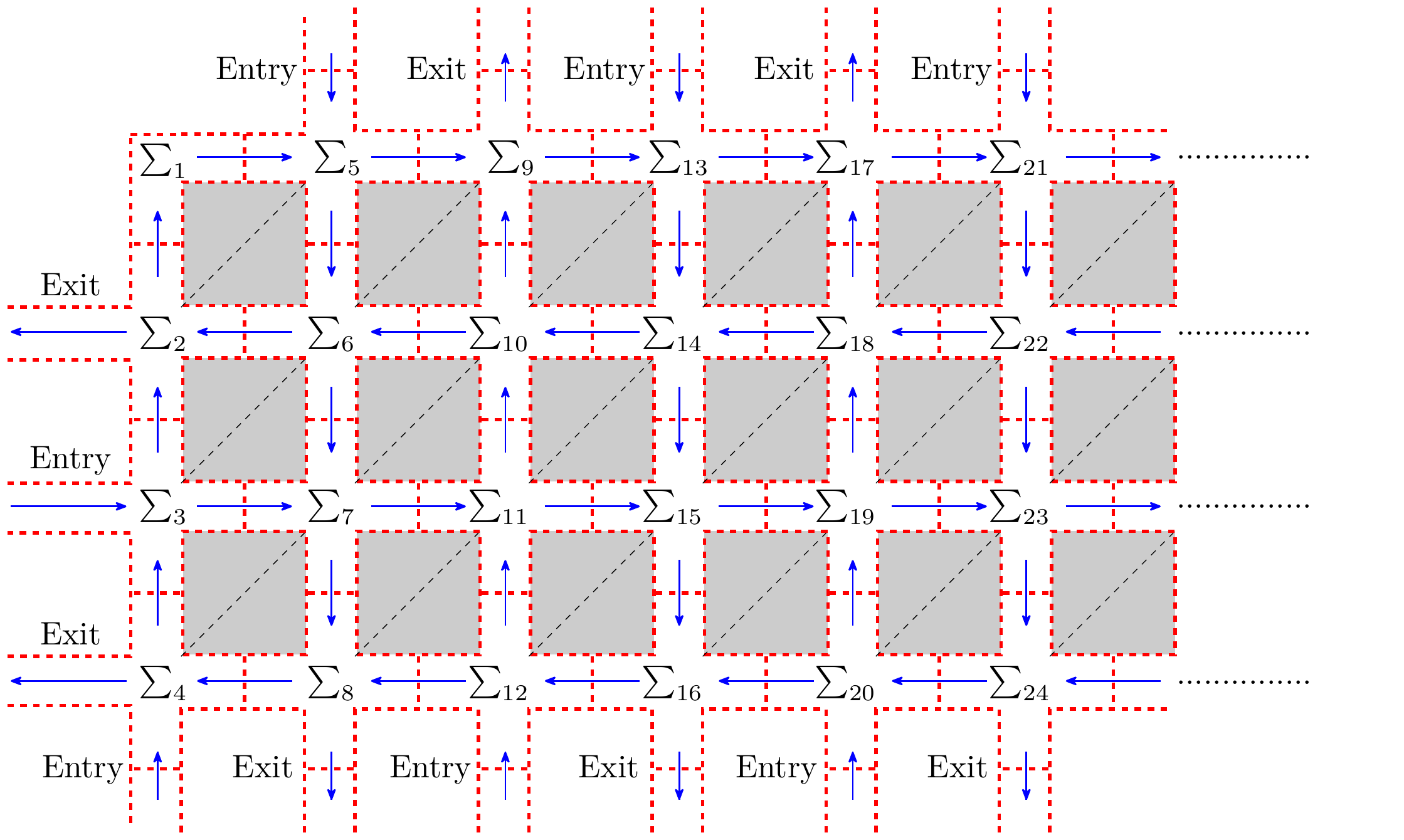}
				\vspace*{-0.6cm}
		\caption{Model of a road traffic network composed of infinitely many subsystems, taken from~\cite{KMS19}.}
		\label{allr}
	\end{center}
	\vspace*{-0.2cm}
\end{figure}

To verify eISS of the network, for each subsystem $\Sigma_i^d$ we take an eISS Lyapunov function of the form $V_i(x_i) = \abs{x_i}$.
The function $V_i$ clearly satisfies~\eqref{eq_viest} and ~\eqref{eq_uniformity_alpha} for all $i\in\N$ with $\ul\omega = \ul{\omega}_i = \ol{\omega}_i = \ol\omega = 1$.
We also have
\begin{align*}
  V_i(x_i^+) & \leq\! \Big(1-T\big(\frac{v_i}{l_i} + e_i\big)\Big) \abs{x_i} \!+\! T c \norm{D_i} \abs{\ol x_{i}}_\infty \!+\! T {B_i} \abs{u_i} \nonumber\\
 & \leq  \max\big\{ \gamma\abs{x_i} , \gamma\abs{\ol x_{i}}_\infty , \frac{1}{\varepsilon} T {B_i} \abs{u_i}\big\} \nonumber\\
  &\leq  \max\big\{ \gamma V_i(x_i) , \gamma V_i(x_{i-1}) , \gamma V_i(x_{i+1}) , \frac{1}{\ep}T {B_i} \abs{u_i} \big\}
\end{align*}
where $\gamma := \Big(1-T\big(\frac{v_i}{l_i} + e_i\big)\Big) + T c \norm{D_i} + \ep$, $\ep > 0$.
This implies that~\eqref{eq_nablaviest} is satisfied with $M=1$, $\gamma_{ij} = \gamma$
for all $j\in \{i-1,i,i+1\}$, $\gamma_{ij} = 0 $ for all $j\in \N \backslash \{i-1,i,i+1\}$, $\gamma_{iu} = T B_i / \ep$.
Additionally one can observe that condition~\eqref{eq_uniformity_gammaiu} is fulfilled with $\ol\gamma_{u} = T r / \ep$.
Finally, condition~\eqref{eq_uniformity_gamma} holds as one can \emph{always}	take $T,c$ and $\ep$ sufficiently small such that
$$\alpha = 1-T\big(\frac{\underaccent{\bar}{v}}{\ol l} \big) + T c\big(\frac{\ol v}{\ul l}\big) + \ep < 1 .$$
We note that all gain functions are linear. 
This together with the previous observations admits the use of Theorem~\ref{MT} to conclude eISS of the network composed of subsystems~\eqref{subsys-d}.

Now by Theorem~\ref{thm:truncation} one see that the performance indices, i.e. the decay rate $\alpha$ and the input gain $\ol\gamma_u$, are preserved for any finite interconnection of $\Sigma^d$.
This is illustrated by Figures~\ref{fig:500} to~\ref{fig:2000}, where we, respectively, consider a network of $10^2$, $10^3$ and $10^4$ cells.
Over the simulation period we take $u_i = 1$ (i.e. let all traffic lights at the entries be green), the sampling period $T =
20$ ms.
Moreover, the initial values are uniformly distributed over $[0,20]$.
From Figures~\ref{fig:500} and~\ref{fig:2000}, the
overall behavior of the network remains almost identical, though the network grows 10 times in size in each case.
This shows the independence of the stability/performances indices from the network size.

\begin{figure}
\vspace*{-0.0cm}
\hspace*{-0.6cm}
\centering
\includegraphics[scale=0.24]{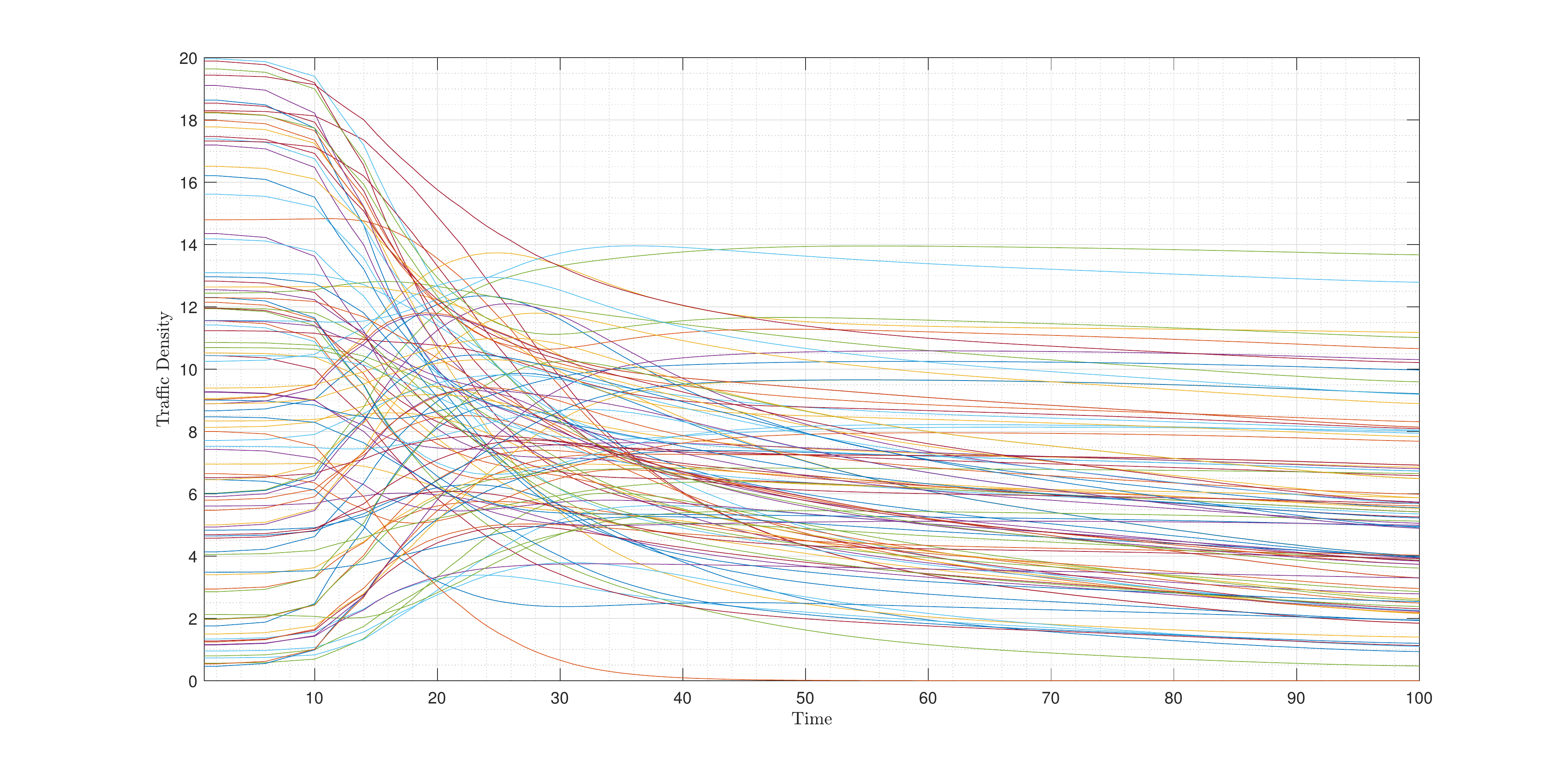}
\vspace*{-0.8cm}
\caption{State trajectories $x_i$ for a network of $10^2$ cells.}
\label{fig:500}
\end{figure}

\begin{figure}
\vspace*{-0.0cm}
\hspace*{-0.6cm}
\centering
\includegraphics[scale=0.24]{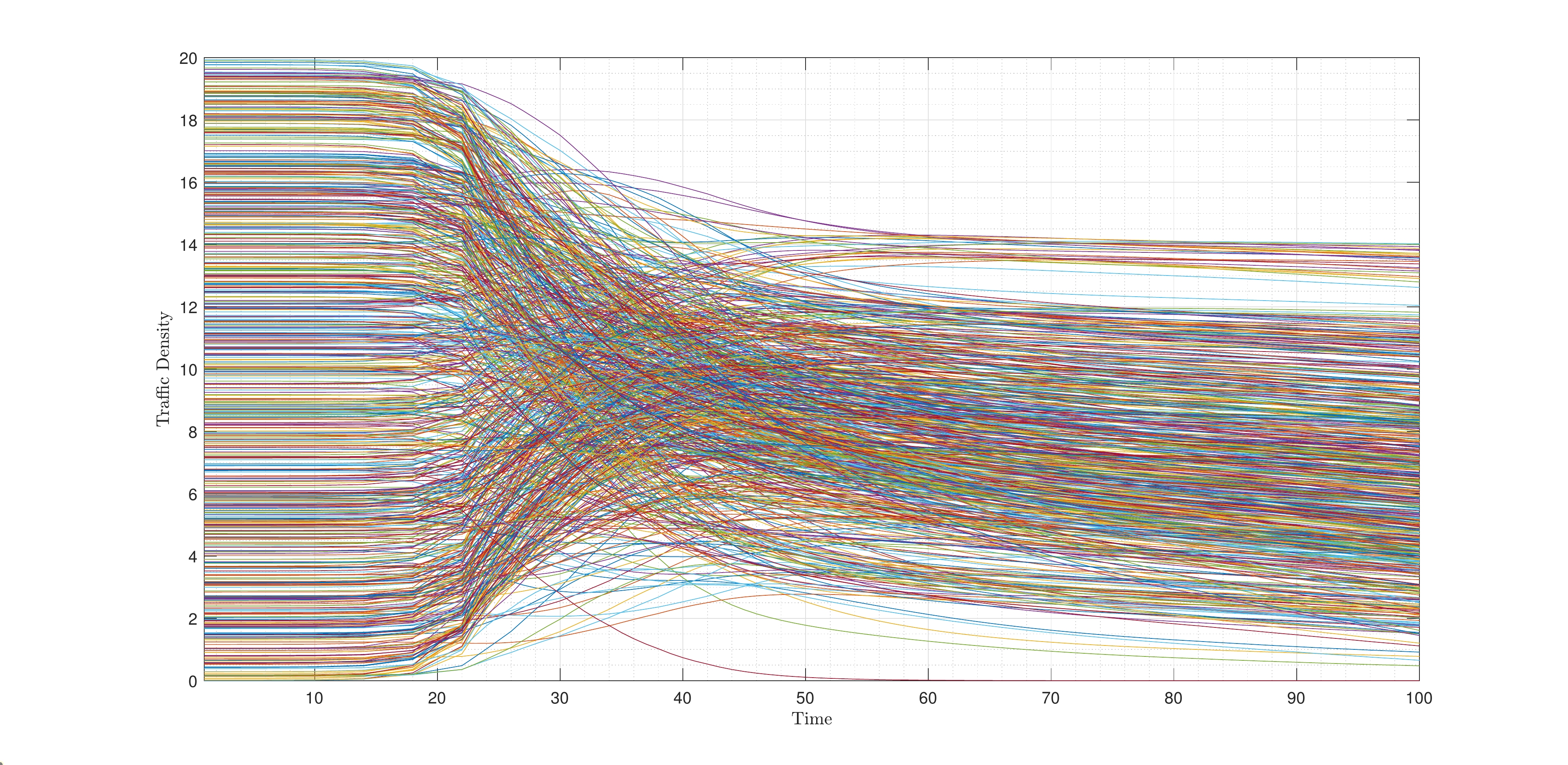}
\vspace*{-0.8cm}
\caption{State trajectories $x_i$ for a network of $10^3$ cells.}
\label{fig:1000}
\end{figure}

\begin{figure}
\vspace*{-0.0cm}
\hspace*{-0.6cm}
\centering
\includegraphics[scale=0.24]{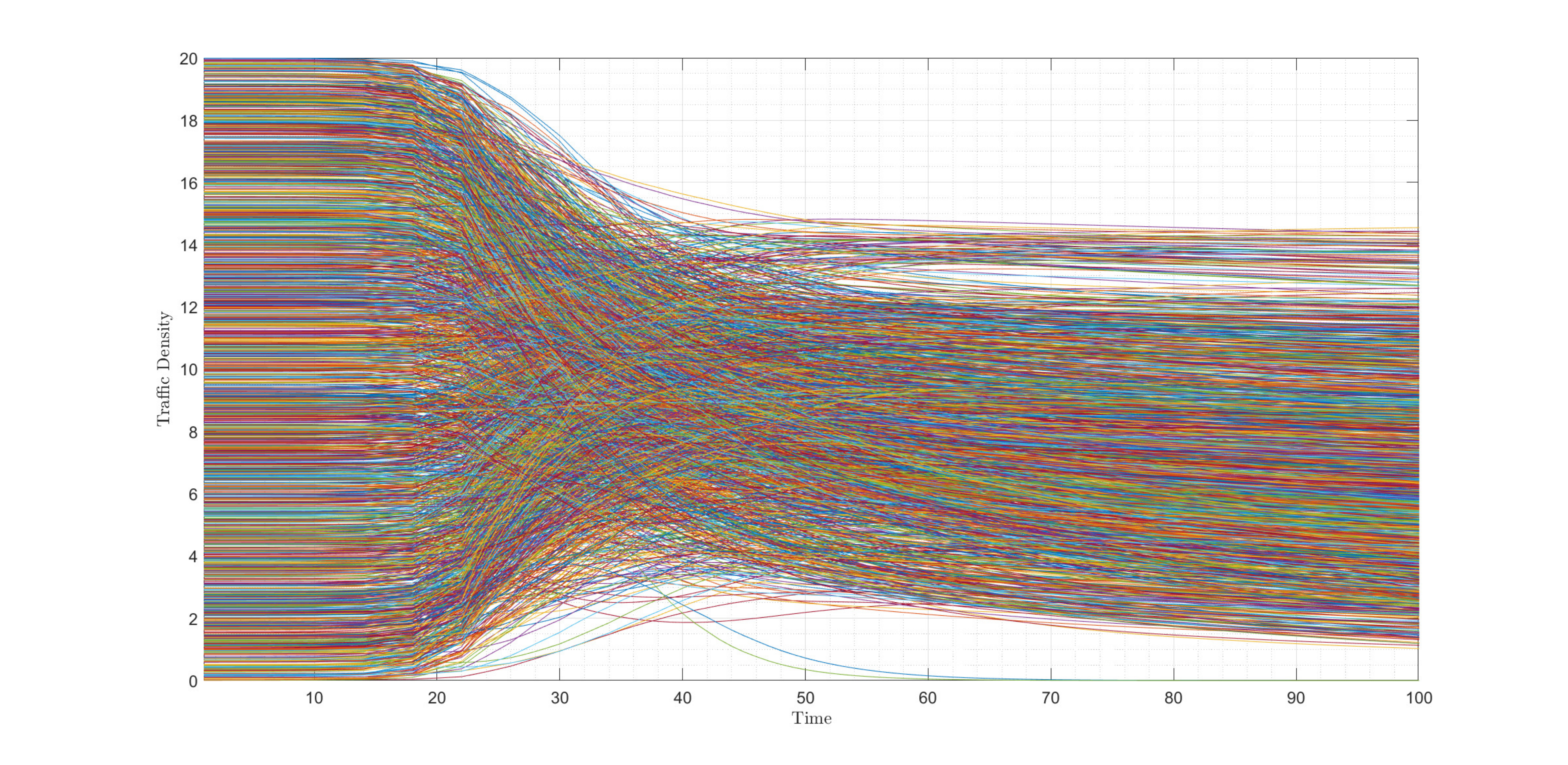}
\vspace*{-0.8cm}
\caption{State trajectories $x_i$ for a network of $10^4$ cells.}
\label{fig:2000}
\end{figure}


\section{Conclusions and outlook} \label{sec:conclusions}

In this paper we have discussed ISS Lyapunov small-gain theorems for discrete time systems given as infinite networks with a locally finite neighborhood structure, where individual subsystems do not have to be ISS.
The necessity of the small-gain condition in case of exponential decay rate of solutions have been established.
It was shown how to use this approach for the over-approximation of large-but-finite networks.
In particular, it has been shown that the ISS property of the infinite network is transferable to any truncation (in size) of the network.

The following challenging problems remain for a future work: in this
work we assumed that all the coupling gains are less than the
identity. However, in view of~\cite{DMS19a,Dashkovskiy.2010,DRW07} this
assumption is not needed for finite networks, where a max-type small-gain formulation is given in the form of the so-called cyclic small-gain condition. Another direction is to investigate the relationship between ISS of an infinite network and that of its truncation with time-varying size.
In view of~\cite{Noroozi.2020a}, we are also investigating application of the results of this paper to distributed control design for infinite networks.

\bibliographystyle{IEEEtran}

\end{document}